\documentclass[preprint,11pt]{elsarticle}    

\usepackage{graphicx,epsfig}
\usepackage{enumerate}
\usepackage{amssymb}
\usepackage{amsthm}
\usepackage{amsmath}
\usepackage{centernot}
\usepackage{xcolor}
\usepackage{todonotes}
\usepackage{times}
\usepackage{mathrsfs}
\usepackage[utf8]{inputenc}
\usepackage{subfigure}
\usepackage{multirow}
\usepackage{setspace}

\usepackage{mathtools}
\usepackage{appendix}

\usepackage{amsmath}
\usepackage{wrapfig}
\usepackage{setspace}
\usepackage{booktabs}
\usepackage{lscape}
\usepackage{algorithm,algorithmicx}
\usepackage{algcompatible}
\usepackage{bm}
\usepackage{graphicx,epsfig}
\usepackage{xfrac}
\usepackage{url}
\usepackage{multirow}
\usepackage{longtable}
\usepackage[linkcolor=blue, urlcolor=blue, citecolor=blue,
colorlinks]{hyperref}

\newtheorem{thm}{Theorem}[section]
\theoremstyle{definition}
\newtheorem{dfn}{Definition}[section]
\newtheorem{lem}{Lemma}[section]
\newtheorem{example}{Example}[section]
\newtheorem{rmrk}{Remark}[]

\makeatletter
\newcommand{\thickhline}{%
	\noalign {\ifnum 0=`}\fi \hrule height 1pt
	\futurelet \reserved@a \@xhline
}

\journal{European Journal of Operational Research }
\begin{document}
\begin{frontmatter}
\title{\textbf{Generalized-Hukuhara Subgradient and its Application in Optimization Problem with Interval-valued Functions}}
\author[iitbhu_math]{Amit Kumar Debnath}
        \ead{amitkdebnath.rs.mat18@itbhu.ac.in}
\author[iitbhu_math]{Debdas Ghosh\corref{cor1}}
        \ead{debdas.mat@iitbhu.ac.in}
\author[radko,radko2]{Radko Mesiar}
        \ead{mesiar@math.sk}
\author[iitbhu_math]{Ram Surat Chauhan}
        \ead{rschauhan.rs.mat16@itbhu.ac.in}
\address[iitbhu_math]{Department of Mathematical Sciences,  Indian Institute of Technology (BHU) Varanasi \\ Uttar Pradesh--221005, India}
\cortext[cor1]{Corresponding author}
\address[radko]{Faculty of Civil Engineering, Slovak University of Technology, Radlinsk\'{e}ho 11, 810 05, Bratislava}
\address[radko2]{Palack\'y University Olomouc, Faculty of Science, Department of Algebra and Geometry, 17. listopadu 12, 771 46 Olomouc, Czech Republic}

\begin{abstract}
In this article, the concepts of $gH$-subgradients and $gH$-subdifferentials of interval-valued functions are illustrated. Several important characteristics of the $gH$-subdifferential of a convex interval-valued function, e.g., closeness, boundedness, chain rule, etc. are studied. Alongside, we prove that $gH$-subdifferential of a $gH$-differentiable convex interval-valued function only contains $gH$-gradient of that interval-valued function. It is observed that the $gH$-directional derivative of a convex interval-valued function in each direction is maximum of all the products of $gH$-subgradients and the direction. Importantly, we show that a convex interval-valued function is $gH$-Lipschitz continuous if it has $gH$-subgradients at each point in its domain. Furthermore, the relations between efficient solutions of an optimization problem with interval-valued function and its $gH$-subgradients are derived.  \\
\end{abstract}

\begin{keyword}
(R) Convex programming\sep $gH$-subgradient \sep $gH$-subdifferential \sep Interval optimization problems\\
\vspace{1.0cm}

\end{keyword}
\end{frontmatter}

\section{\textbf{Introduction}}

\noindent In the real-life decision-making process, we often face the optimization problem with nonsmooth functions and in order to deal with the optimization problems with nonsmooth functions, the concepts of subgradient and subdifferential inevitably arise. Due to inexact and imprecise natures of many real-world occurrences the study of Interval-Valued Functions (IVFs) and Optimization problems with IVFs, known as Interval Optimization Problems (IOP)s, become substantial topics to the researchers. In this article, we illustrate the concepts of subgradient and subdifferential for IVFs, and study the several important characteristics of subgradient and subdifferential of IVFs. We also study the optimality conditions for nonsmooth IOPs. As intervals are the inextricable things in IVFs and IOPs, before making a survey on IVFs and IOPs, we make a survey on the arithmetic and ordering of intervals.

\subsection{Literature Survey}

In the literature of IVFs, to deal with compact intervals and IVFs, Moore \cite{Moore1966}  developed interval arithmetic. There are a few limitations (see \cite{Ghosh2019derivative} for details) of Moore's interval arithmetic; especially, Moore's interval arithmetic cannot provide the additive inverse of a nondegenerate interval. By a nondegenerate interval, we mean an interval whose upper and lower limits are different. To overcome this difficulty, Hukuhara \cite{Hukuhara1967} proposed a new rule for the difference of intervals, known as Hukuhara difference of intervals. Although the Hukuhara difference provides the additive inverse of any compact interval, it is not applicable between all pairs of compact intervals (see \cite{Ghosh2019derivative} for details). For this reason, the `nonstandard subtraction', introduced by Markov \cite{Markov1979}, has been used and named as generalized Hukuhara difference ($gH$-difference) by Stefanini \cite{Stefanini2008}. The generalized Hukuhara difference is applicable for all pairs of compact intervals and it also provides the additive inverse of any compact interval.\\

Unlike the real numbers, intervals are not linearly ordered. Isibuchi and Tanaka \cite{Ishibuchi1990}  suggested a few partial ordering relations of intervals. In  \cite{Bhurjee2012} some ordering relations based on the parametric representation of intervals are proposed. Also, an ordering relation of intervals is provided in \cite{Costa2015} by a bijective map from the set of intervals to $\mathbb{R}^2$. However, all the ordering relations of \cite{Bhurjee2012, Costa2015} can be derived from the ordering relations of \cite{Ishibuchi1990}. The concept of variable ordering relation of intervals is introduced in \cite{Ghosh2020ordering}.\\

Calculus is one of the most important tools in functional analysis. Therefore, as like the real-valued, vector-valued functions, the development of the calculus for IVFs is much more essential to study the characteristics of IVFs. In order to develop the calculus of IVFs, the concept of differentiability of IVFs was initially introduced by Hukuhara  \cite{Hukuhara1967} with the help of Hukuhara difference of intervals. However, this definition of Hukuhara differentiability is restrictive \cite{Chalco2013-2}. Based on $gH$-difference, the concepts of $gH$-derivative, $gH$-partial derivative, $gH$-gradient, and $gH$-differentiability for IVFs are provided in \cite{Chalco2013kkt, Ghosh2016newton, Markov1979, Stefanini2009, Stefanini2019}. Lupulescu studied the differentiability and the integrability for the IVFs on time scales in \cite{Lupulescu2013} and developed the fractional calculus for IVFs in \cite{Lupulescu2014}. The concept of directional $gH$-derivative for IVF is depicted in \cite{Bao2016, Stefanini2019}. Recently, Ghosh et al.\ \cite{Ghosh2019derivative} have introduced the idea of $gH$-G\^{a}teaux derivative, and $gH$-Fr\'echet derivative of IVFs.\\

Based on the existing ordering relations of intervals and calculus of IVFs many authors developed the theories to characterize the solutions to IOPs. For instance, using the concept of Hukuhara differentiability, Wu proposed Karush-Kuhn-Tucker (KKT) conditions for IOPs in \cite{Wu2007}. In \cite{Wu2008}, Wu presented the solution concepts of IOPs with the help of bi-objective optimization. Also, Wu presented some duality conditions for IOPs in \cite{Wu2008duality, Wu2010}. Using the concept of $gH$-differentiability, Chalco-Cano et al. \cite{Chalco2013kkt} presented KKT conditions for IOPs. Ghosh et al. \cite{Ghosh2019extended} developed generalized KKT conditions to obtain the solution of the IOPs. 
Recently, Stefanini et al. \cite{Stefanini2019} have depicted the optimality conditions for IOPs using the concepts of directional $gH$-derivative and total $gH$-derivative of IVFs, and  Ghosh et al. have developed the optimality conditions for IOPs using the concepts of $gH$-G\^{a}teaux derivative, and $gH$-Fr\'echet derivative of IVFs.\\

The authors of \cite{Ahmad2017, Antczak2017, Jayswal2016, Van2018, Van2020} proposed various optimality and duality conditions for nonsmooth IOPs converting them into real-valued multiobjective optimization. However, in this approach, one needs the closed-form of boundary functions of the interval-valued objective and constrained functions as readily available, which is practically difficult. Because even for a very simple IVF $\textbf{F}$, the closed forms of the lower boundary function $\underline{f}$ and upper boundary function $\overline{f}$ are not easy to execute; for instance, consider $\textbf{F}(x_1, x_2) = \tfrac{[-1, 6] \odot x_1 \oplus [3, 5]\odot x_2}{ [-2, 7] \odot x_1 \oplus [-4, 0]\odot x_2}$ for all $(x_1, x_2)\in \mathbb{R}^2$. Apart from these, based on parametric representations of the IVFs, some authors \cite{Bhurjee2012, Ghosh2016newton, Ghosh2017quasinewton} studied IOPs and developed theories to obtain the solutions to IOPs by converting them into real-valued optimization problems. The authors of \cite{Bhurjee2016} proposed some optimality conditions and duality results of a nonsmooth convex IOP using the parametric representation of its interval-valued objective and constrained functions. However, the parametric process is also practically difficult. Because in the parametric process, the number of variables increases with the number of intervals involved in the IVFs, and to verify any property of an IVF one has to verify it for an infinite number of its corresponding real-valued function. For instance, see Definition $9$ in \cite{Bhurjee2016}.

\subsection{Motivation and Contribution}
 From the literature of IVFs and IOPs, it is observed that the concepts of subgradient and subdifferentials for IVFs are not properly introduced. However, the authors of \cite{Hai2018} proposed the concepts of subgradient and subdifferentials for $n$-cell convex fuzzy-valued functions (FVFs) and proved that the subdifferentials of convex FVFs are convex. But they didn't mention about other important properties of subgradient and subdifferentials of FVFs, such as closeness, boundedness, chain rule, etc. of subdifferentials. As IVFs are the special case of FVFs, in this article, at first adopting the concept of subgradient for convex FVFs of the article \cite{Hai2018} we define subgradient of convex IVFs (namely $gH$-subgradient). Thereafter, we illustrate the concept of subgradient for convex IVFs in terms of linear IVFs. Subsequently, we define subdifferential of convex IVFs (namely $gH$-subdifferential) and study its various properties. We prove that $gH$-subdifferentials of convex IVFs are closed, and bounded sets. In order to prove these properties, the norm on the set of $gH$-continuous bounded linear IVFs is defined and the idea of sequences with their convergence on the set of $n$-tuple of compact intervals is described. Although the author of \cite{Karaman2020} provided the concept of subgradients for IVFs in terms of linear functions, our concept is more general (please see Remark \ref{rsg} of this article for details).\\

 Along with the aforementioned properties of $gH$-subdifferentials, several important characteristics of $gH$-subgradients are also studied in this article. Interestingly, it is observed that if a convex IVF has $gH$-subgradients at each point in its domain, then the IVF is  $gH$-Lipschitz continuous. It is reported that the $gH$-directional derivative of a convex IVF is the maximum of the products of the $gH$-subgradients and the concerning direction. The chain rule of a convex IVF and the $gH$-subgradient of the sum of finite numbers of convex IVFs are illustrated. Also, some optimality conditions of nonsmooth convex IOP \emph{without applying the parametric approach} are explored in this article. Most importantly, it is to mention that all the proposed definitions and the results of this article are applicable to general IVFs regardless of whether or not
\begin{enumerate}[(i)]
\item the IVFs can be expressed parametrically, or
\item the explicit form of the lower and upper boundary functions of the IVFs can be found.
\end{enumerate}

\subsection{Delineation}
 The proposed work is organized as follows. The next section deals with some basic terminologies and notions of intervals analysis followed by the convexity and calculus of IVFs. The concepts of $gH$-subgradients and $gH$-subdifferentials of IVFs with their several important characteristics are illustrated in Section \ref{ssd}. It is shown that the $gH$-subdifferential of a convex IVF is closed and bounded. It is observed that a $gH$-differentiable convex IVF has only one $gH$-subgradient. It is also proved that the $gH$-directional derivative of a convex interval-valued function in each direction is maximum of all the products of $gH$-subgradients and the direction. Further in Section \ref{ssd}, it is shown that a convex IVF is $gH$-Lipschitz continuous if it has $gH$-subgradients at each point in its domain. Apart from these, The chain rule of a convex IVF and the $gH$-subgradient of the sum of finite numbers of convex IVFs are illustrated. The relations between efficient solutions of an IOP with $gH$-subgradients of its objective function are derived in Section \ref{siop}. Finally, the last section is concerned with a few future directions for this study.

\section{\textbf{Preliminaries and Terminologies}}
	
\noindent This section is devoted to some basic terminologies and notions on intervals. Convexity and calculus of IVFs are also described here. The ideas and notations that we describe in this section are used throughout the paper.

\subsection{Interval Arithmetic, Dominance Relation and Sequence of Intervals}\label{ssai}
%
%
%
	
Let $\mathbb{R}$ be the set of real numbers, $\mathbb{R}_+$ be the set of all nonnegative real numbers, and $I(\mathbb{R})$ be the set of all compact intervals. We denote the elements of $I(\mathbb{R})$ by bold capital letters ${\textbf A}, {\textbf B}, {\textbf C}, \cdots $. We represent an element $\textbf{A}$ of $I(\mathbb{R})$ with the help of corresponding small letter in the following way
\[
\textbf{A} = [\underline{a}, \overline{a}].
\]
Similarly, ${\textbf B} = [\underline{b}, \overline{b}]$, ${\textbf C} = [\underline{c}, \overline{c}]$, and so on.  It is to note that any singleton set $\{p\}$ of $\mathbb{R}$ can be represented by the interval $[\underline{p},\;\overline{p}]$ with  $\underline{p}=p=\overline{p}$. In particular, $\textbf{0}=\{0\}=[0, 0]$.\\

In this article, along with the Moore's interval addition ($\oplus$), substraction ($\ominus$), multiplication ($\odot$), and division ($\oslash$) \cite{Moore1966,Moore1987}:
\begin{align*}
&\textbf{A} \oplus \textbf{B} = \left[\underline{a} + \underline{b}, \overline{a} +
\overline{b}\right],~\textbf{A} \ominus \textbf{B} = \left[\underline{a} - \overline{b}, \overline{a} -
\underline{b}~\right],\\
&\textbf{A} \odot \textbf{B}  = \left[\min\left\{\underline{a}\underline{b}, \underline{a}\overline{b},\overline{a}\underline{b}, \overline{a}\overline{b}\right\}, \max\left\{\underline{a}\underline{b}, \underline{a}\overline{b},\overline{a}\underline{b}, \overline{a}\overline{b}\right\} \right],\\
&\textbf{A} \oslash \textbf{B} = \left[\min\left\{\underline{a}/\underline{b}, \underline{a}/\overline{b}, \overline{a}/\underline{b}, \overline{a}/\overline{b}\right\}, \max\left\{\underline{a}/\underline{b}, \underline{a}/\overline{b}, \overline{a}/\underline{b}, \overline{a}/\overline{b}\right\} \right], \text{ provided } 0\not\in \textbf{B},
\end{align*}
we use $gH$-difference ($\ominus_{gH}$) of intervals because $\textbf{A}\ominus\textbf{A}\neq\textbf{0}$ for a nondegenerate interval $\textbf{A}$. The $gH$-difference \cite{Markov1979, Stefanini2008} of the interval $\textbf{B}$ from the interval $\textbf{A}$ is defined by the interval $\textbf{C}$ such that
\[
\textbf{A} =  \textbf{B} \oplus  \textbf{C} ~\text{ or }~ \textbf{B} = \textbf{A}
\ominus \textbf{C}.
\]
It is to be noted that for $\textbf{A} = \left[\underline{a},\overline{a}\right]$ and $\textbf{B} = \left[\underline{b},\overline{b}\right]$,
\[
\textbf{A} \ominus_{gH} \textbf{B} = \left[\min\{\underline{a}-\underline{b},
\overline{a} - \overline{b}\}, \max\{\underline{a}-\underline{b}, \overline{a} -
\overline{b}\}\right]~\text{and}~\textbf{A} \ominus_{gH} \textbf{A} = \textbf{0}.
\]
\begin{rmrk}\label{ria1}
It is easy to check that the addition of intervals are commutative and associative, and
\[
\textbf{A} \ominus \textbf{B}=\textbf{A} \oplus (-1)\odot\textbf{B}.
\]
\end{rmrk}
The algebraic operations on the product space $I(\mathbb{R})^n=I(\mathbb{R})\times I(\mathbb{R})\times \cdots \times I(\mathbb{R})$ ($n$ times) are defined as follows.
\begin{dfn}(\emph{Algebraic operations on $I(\mathbb{R})^n$})\label{daoirn}. Let $\widehat{\textbf{A}} = \left(\textbf{A}_1, \textbf{A}_2, \cdots, \textbf{A}_n\right)$ and $\widehat{\textbf{B}} = (\textbf{B}_1, \textbf{B}_2,$ $\cdots, \textbf{B}_n)$ be two elements of $I(\mathbb{R})^n$. An algebraic operation $\star$ between $\widehat{\textbf{A}}$ and $\widehat{\textbf{B}}$, denoted by $\widehat{\textbf{A}} \star \widehat{\textbf{B}}$, is defined by
\[
\widehat{\textbf{A}} \star \widehat{\textbf{B}}=\left(\textbf{A}_1\star \textbf{B}_1, \textbf{A}_2\star \textbf{B}_2, \cdots, \textbf{A}_n\star \textbf{B}_n\right),
\]
where $\star\in\{\oplus,\ \ominus, \ \ominus_{gH}\}$.
\end{dfn}
The authors of \cite{Ishibuchi1990} defined the ordering relations of intervals of the following types `$\leq_{LR}$', `$\leq_{CW}$', and `$\leq_{LC}$'. In this article, we only use `$\leq_{LR}$' ordering relation and simply denote by `$\preceq$'. Also, it is to mention that in view of the ordering relation `$\preceq$', we define the dominance relations of intervals as follows.
\begin{dfn}({\it Dominance relations on intervals}). Let $\textbf{A}$ and $\textbf{B}$ be two intervals in $I(\mathbb{R})$.
\begin{enumerate}[(i)]
\item $\textbf{B}$ is said to be dominated by $\textbf{A}$ if $\underline{a}\leq \underline{b}$ and $\overline{a}\leq\overline{b}$, and then we write $\textbf{A}\preceq \textbf{B}$;
\item $\textbf{B}$ is said to be strictly dominated by $\textbf{A}$ if either $\underline{a} \leq \underline{b}$  and $\overline{a} < \overline{b}$ or $\underline{a} < \underline{b}$  and $\overline{a} \leq \overline{b}$, and then we write $\textbf{A}\prec \textbf{B}$;
\item if $\textbf{B}$ is not dominated by $\textbf{A}$, then we write $\textbf{A}\npreceq \textbf{B}$; if $\textbf{B}$ is not strictly dominated by $\textbf{A}$, then we write $\textbf{A}\nprec \textbf{B}$;
\item if $\textbf{A}\npreceq \textbf{B}$ and $\textbf{B}\npreceq \textbf{A}$, then we say that none of $\textbf{A}$ and $\textbf{B}$ dominates the other, or $\textbf{A}$ and $\textbf{B}$ are not comparable.
\end{enumerate}
\end{dfn}
Now we illustrate the concept of sequence in $I(\mathbb{R})^n$ and study its convergence. To do so, we need the concepts of norm on $I(\mathbb{R})$ as well as on $I(\mathbb{R})^n$.
\begin{dfn}\label{irnorm}
(\emph{Norm on $I(\mathbb{R})$} \cite{Moore1966}). For an $\textbf{A} = \left[\underline{a}, \bar{a}\right]$ in $ I(\mathbb{R})$, the function ${\lVert \cdot \rVert}_{I(\mathbb{R})} : I(\mathbb{R}) \rightarrow \mathbb{R}_+$, defined by
\[
{\lVert \textbf{A} \rVert}_{I(\mathbb{R})} = \max \{|\underline{a}|, |\bar{a}|\},
\]
is a norm on $I(\mathbb{R})$.
\end{dfn}
\begin{dfn}\label{irnnorm}(\emph{Norm on $I(\mathbb{R})^n$}). For an $\widehat{\textbf{A}} = \left(\textbf{A}_1, \textbf{A}_2, \cdots, \textbf{A}_n\right)\in I(\mathbb{R})^n$, the function ${\lVert \cdot \rVert}_{I(\mathbb{R})^n} : I(\mathbb{R}) \rightarrow \mathbb{R}_+$, defined by
\[
{\lVert \widehat{\textbf{A}} \rVert}_{I(\mathbb{R})^n} = \sqrt{\sum_{i=1}^n {\lVert \textbf{A}_i \rVert}_{I(\mathbb{R})}^2},
\]
is a norm on $I(\mathbb{R})^n$. To prove that the function ${\lVert \cdot \rVert}_{I(\mathbb{R})^n}$ satisfies all the properties of a norm please see \ref{apirnnorm}.
\end{dfn}
In the rest of the article, we use the symbols `${\lVert \cdot \rVert}_{I(\mathbb{R})}$' and `${\lVert \cdot \rVert}_{I(\mathbb{R})^n}$' to denote the norms on $I(\mathbb{R})$ and $I(\mathbb{R})^n$, respectively, but we simply use the symbol `${\lVert \cdot \rVert}$' to denote the usual Euclidean norm on $\mathbb{R}^n$.

\begin{dfn}(\emph{Sequence in $I(\mathbb{R})^n$}). A function $\widehat{\textbf{G}}:\mathbb{N} \rightarrow I(\mathbb{R})^n$ is called sequence in $I(\mathbb{R})^n$.
\end{dfn}
\begin{dfn}(\emph{Bounded sequence in $I(\mathbb{R})^n$}). A sequence $\left\{\widehat{\textbf{G}}_k\right\}$ in $I(\mathbb{R})^n$ is said to be bounded from below  (above) if there exists an $\widehat{\textbf{A}}\in I(\mathbb{R})^n$ (a $\widehat{\textbf{B}}\in I(\mathbb{R})^n$) such that
\[
\widehat{\textbf{A}}\preceq\widehat{\textbf{G}}_k ~ \text{for all} ~n\in \mathbb{N} ~(\widehat{\textbf{G}}_k\preceq\widehat{\textbf{B}}  ~\text{for all} ~n\in \mathbb{N}),
\]
where for any two elements $\widehat{\textbf{C}}=\left(\textbf{C}_{1}, \textbf{C}_{2}, \cdots, \textbf{C}_{n}\right)$ and $\widehat{\textbf{D}}=\left(\textbf{D}_{1}, \textbf{D}_{2}, \cdots, \textbf{D}_{n}\right)$ in $I(\mathbb{R})^n$,
\[
\widehat{\textbf{C}}\preceq\widehat{\textbf{D}}\Longleftrightarrow \textbf{C}_i\preceq\textbf{D}_i \text{ for all } i=1, 2, \cdots, n.
\]
A sequence $\left\{\widehat{\textbf{G}}_k\right\}$ that is both bounded below and above is called a bounded sequence.
\end{dfn}
\begin{dfn}(\emph{Convergence in $I(\mathbb{R})^n$}).
A sequence $\left\{\widehat{\textbf{G}}_k\right\}$ in $I(\mathbb{R})^n$ is said to be convergent if there exists a $\widehat{\textbf{G}}\in I(\mathbb{R})^n$ such that
\[\lVert  \widehat{\textbf{G}}_k \ominus_{gH} \widehat{\textbf{G}}  \rVert_{I(\mathbb{R})^n} \to 0~\text{as}~k \to \infty .\]
\end{dfn}
\begin{rmrk}\label{ncs}
It is noteworthy that if a sequence $\left\{\widehat{\textbf{G}}_k\right\}$ in $I(\mathbb{R})^n$, where $\widehat{\textbf{G}}_k = (\textbf{G}_{k1}, \textbf{G}_{k2}, \cdots,$ $\textbf{G}_{kn})$, converges to $\widehat{\textbf{G}}=\left(\textbf{G}_{1}, \textbf{G}_{2}, \cdots, \textbf{G}_{n}\right) \in I(\mathbb{R})^n$, then according to Definition \ref{daoirn} and Definition \ref{irnnorm}, corresponding each sequence $\left\{\textbf{G}_{ki}\right\}$ in $I(\mathbb{R})$ converges to $\textbf{G}_i \in I(\mathbb{R})$ for all $i= 1, 2, \cdots, n $. Also, due to Definition \ref{irnorm}, the sequences $\left\{\underline{g_{ki}}\right\}$ and $\left\{\overline{g_{ki}}\right\}$ in $\mathbb{R}$ converge to $\underline{g_i}$ and $\overline{g_i}$, respectively, for all $i$.
\end{rmrk}

\subsection{Convexity and Calculus of IVFs}

A function $\textbf{F}$ from a nonempty subset $\mathcal{X}$ of $\mathbb{R}^n$ to $I(\mathbb{R})$ is known as an IVF (interval-valued function). For each argument point $x \in \mathcal{X}$, the value of $\textbf{F}$ is presented by
\[
\textbf{F}(x)=\left[\underline{f}(x),\
\overline{f}(x)\right],
\]
where $\underline{f}(x)$ and $\overline{f}(x)$ are real valued functions on $\mathcal{X}$ such that $\underline{f}(x)\leq \overline{f}(x)$ for all $x\in\mathcal{X}$.\\

In \cite{Wu2007}, Wu introduced two types of convexity for IVF, i.e., `LU-convexity' and `UC-convexity'. However, in this article, we only use LU-convexity for IVF and we read an LU-convex IVF as simply a convex IVF, which is defined follows.
\begin{dfn}(\emph{Convex IVF} \cite{Wu2007})\label{dcivf}. Let $\mathcal{X} \subseteq \mathbb{R}^n$ be a convex set. An IVF $\textbf{F}: \mathcal{X} \rightarrow I(\mathbb{R})$ is said to be a convex IVF if for any two vectors $x_1$ and $x_2$ in $\mathcal{X}$,
\[
\textbf{F}(\lambda_1 x_1+\lambda_2 x_2)\preceq
\lambda_1\odot\textbf{F}(x_1)\oplus\lambda_2\odot\textbf{F}(x_2)
\]
for all $\lambda_1,~\lambda_2\in[0,\ 1]$ with $\lambda_1+\lambda_2=1$.
\end{dfn}
It is notable that in Definition \ref{dcivf}, we have used the notation `$\preceq$' instead of `$\preceq_{LC}$'. Because the ordering relation `$\preceq_{LC}$' provided in \cite{Wu2007} is same as the ordering relation `$\preceq$'.
\begin{lem}\label{lc1}(See \cite{Wu2007}).
$\textbf{F}$ is convex if and only if $\underline{f}$
and $\overline{f}$ are convex.
\end{lem}
\begin{dfn}(\emph{$gH$-continuity} \cite{Ghosh2016newton, Markov1979}).
Let $\textbf{F}$ be an IVF on a nonempty subset $\mathcal{X}$ of $\mathbb{R}^n$. Let $\bar{x}$ be an interior point of $\mathcal{X}$
and $d\in\mathbb{R}^n$ be such that $\bar{x}+d\in\mathcal{X}$. The function
$\textbf{F}$ is said to be a $gH$-continuous at $\bar{x}$ if
\[
\lim_{\lVert d \rVert\rightarrow 0}\left(\textbf{F}(\bar{x}+d)\ominus_{gH}\textbf{F}(\bar{x})\right)=\textbf{0}.
\]
\end{dfn}
\begin{lem}\label{lc2}
An IVF $\textbf{F}$ on a nonempty subset $\mathcal{X}$ of $\mathbb{R}^n$ is $gH$-continuous if and only if
$\underline{f}$ and $\overline{f}$ are continuous.
\end{lem}
\begin{proof}
Please see \ref{aplc2}.
\end{proof}
\begin{thm}\label{tc}
If an IVF $\textbf{F}$ on a nonempty open convex subset $\mathcal{X}$ of $\mathbb{R}^n$ is convex, then $\textbf{F}$ is $gH$-continuous on $\mathcal{X}$.
\end{thm}
\begin{proof}
Please see \ref{aptc}.
\end{proof}
\begin{dfn}(\emph{$gH$-Lipschitz continuous IVF} \cite{Ghosh2019derivative}). Let $\mathcal{X}\subseteq \mathbb{R}^n$. An IVF  $\textbf{F}: \mathcal{X} \rightarrow I(\mathbb{R})$ is said to be $gH$-Lipschitz continuous on $\mathcal{X}$ if there exists $L~>~0 $ such that
\[ {\lVert \textbf{F}(x) \ominus_{gH} \textbf{F}(y) \rVert }_{I(\mathbb{R})} \le L {\lVert x-y \rVert} ~\text{for all}~x,y \in \mathcal{X}. \]
The constant $L$ is called a Lipschitz constant.
\end{dfn}
\begin{dfn}(\emph{$gH$-derivative} \cite{Markov1979, Stefanini2009}).
Let $\mathcal{X}\subseteq \mathbb{R}$. The $gH$-derivative of an IVF $\textbf{F}:\mathcal{X} \rightarrow I(\mathbb{R})$ at $\bar{x}\in \mathcal{X}$ is defined by
\[
\textbf{F}'(\bar{x})=\displaystyle\lim_{d\rightarrow 0} \frac{\textbf{F}(\bar{x}+d) \ominus_{gH} \textbf{F}(\bar{x})}{d},~\text{provided the limit exists.}
\]
\end{dfn}
\begin{rmrk} (See \cite{Chalco2011, Markov1979}).\label{rd1}
Let $\mathcal{X}$ be a nonempty subset of $\mathbb{R}$. The $gH$-derivative of an IVF $\textbf{F}:\mathcal{X} \rightarrow I(\mathbb{R})$ at $\bar{x}\in \mathcal{X}$ exists if the derivatives of $\underline{f}$ and $\overline{f}$ at $\bar{x}$ exist and
\[
\textbf{F}'(\bar{x})=\left[\min\left\{\underline{f}'(\bar{x}), \overline{f}'(\bar{x})\right\}, \max\left\{\underline{f}'(\bar{x}), \overline{f}'(\bar{x})\right\}\right].
\]
However, the converse is not true.
\end{rmrk}
\begin{dfn}\label{pdgh}(\emph{Partial $gH$-derivative} \cite{Chalco2013kkt, Ghosh2016newton}).
Let $\textbf{F}:\mathcal{X} \rightarrow I(\mathbb{R})$ be an IVF, where $\mathcal{X}$ is a nonempty subset of $\mathbb{R}^n$. We define a function $\textbf{G}_i$ by
\[
\textbf{G}_i (x_i) = \textbf{F} (\bar{x}_1, \bar{x}_2, \cdots, \bar{x}_{i-1}, x_i, \bar{x}_{i+1}, \cdots, \bar{x}_n),
\]
where $\bar{x} = (\bar{x}_1,\, \bar{x}_2,\, \cdots,\, \bar{x}_n)^T\in\mathcal{X}$. If the $gH$-derivative of $\textbf{G}_i$ exists at $\bar{x}_i$, then the $i$-th partial $gH$-derivative of $\textbf{F}$ at $\bar{x}$, denoted $D_i \textbf{F}(\bar{x})$, is defined by
\[
D_i \textbf{F}(\bar{x})=\textbf{G}'_i (\bar{x}_i)~\text{for all}~i = 1,\, 2,\, \cdots,\, n.
\]
\end{dfn}
\begin{dfn} (\emph{$gH$-gradient} \cite{Chalco2013kkt, Ghosh2016newton}).
Let $\mathcal{X}$ be a nonempty subset of $\mathbb{R}^n$. The $gH$-gradient of an IVF $\textbf{F}:\mathcal{X} \rightarrow I(\mathbb{R})$ at a point $\bar{x} \in \mathcal{X}$, denoted $\nabla \textbf{F} (\bar{x})$, is defined by
\[
\nabla \textbf{F} (\bar{x})=\left(   D_1\textbf{F}(\bar{x}),\,
         D_2\textbf{F}(\bar{x}),\, \cdots,\,
         D_n\textbf{F}(\bar{x})
\right)^T.
\]
\end{dfn}
\begin{dfn}\label{ddd}
(\emph{Directional $gH$-derivative} \cite{Bao2016, Stefanini2019}).
Let $\textbf{F}$ be an IVF on a nonempty subset $\mathcal{X}$ of $\mathbb{R}^n$. Let $\bar{x} \in \mathcal{X}$ and $h \in \mathbb{R}^n$ such that $\bar{x}+\lambda h\in \mathcal{X}$ for any small $\lambda$. If the limit
\[
\lim_{\lambda \to 0+}\frac{1}{\lambda}\odot\left(\textbf{F}(\bar{x}+\lambda h)\ominus_{gH}\textbf{F}(\bar{x})\right)
\]
exists, then the limit is said to be directional $gH$-derivative of $\textbf{F}$ at $\bar{x}$ in the direction $h$, and it is denoted by $\textbf{F}'(\bar{x})(h)$.
\end{dfn}
\begin{dfn}(\emph{Linear IVF} \cite{Ghosh2019derivative}). \label{dlivf}
Let $\mathcal{Y}$ be a linear subspace of $\mathbb{R}^n$. The function $\textbf{L}: \mathcal{Y} \rightarrow I(\mathbb{R})$ is said to be linear if
\begin{enumerate}[(i)]
\item\label{cl1} $\textbf{L}(\lambda x)=\lambda\odot\textbf{L}(x)~ \text{for all}~x\in \mathcal{X}~\text{and for all}~\lambda \in \mathbb{R}$,
\item\label{cl2} for all $x,~y\in \mathcal{Y}$, either
\[
\textbf{L}(x)\oplus\textbf{L}(y) = \textbf{L}(x+y)
\]
 or none of $\textbf{L}(x)\oplus\textbf{L}(y)$ and $\textbf{L}(x+y)$ dominates the other.
\end{enumerate}
\end{dfn}
\begin{rmrk} (See \cite{Ghosh2019derivative}).\label{nlivf}
The IVF $\textbf{L}: \mathbb{R}^n \rightarrow I(\mathbb{R})$ that is defined by
\[
\textbf{L}(x)=d^T\odot\widehat{\textbf{A}}=\bigoplus_{i=1}^n x_i \odot \textbf{A}_i=\bigoplus_{i=1}^n x_i \odot [\underline{a}_i, \overline{a}_i]
\]
is a linear IVF, where `$\bigoplus_{i=1}^n$' denotes successive addition of $n$ number of intervals.
\end{rmrk}
We denote the set of all $gH$-continuous linear IVF on a linear space $\mathcal{Y}\subset \mathbb{R}^n$ as $\widehat{\mathcal{Y}}$.
\begin{dfn}\label{kkk}
(\emph{Bounded linear interval-valued operator} \cite{Ghosh2019derivative}). Let $\mathcal{Y}$ be a real normed space. A linear IVF  $\mathbf{L}: \mathcal{Y} \rightarrow I(\mathbb{R})$ is said to be a bounded linear operator if there exists $K>0$ such that
\[ {\lVert \mathbf{L}(x) \rVert }_{I(\mathbb{R})} \leq K {\lVert x \rVert}_\mathcal{X} ~\text{for all}~ x \in \mathcal{Y}. \]
\end{dfn}
\begin{lem}(See  \cite{Ghosh2019derivative}).\label{kk60}
Let $\mathcal{Y}$ be a real normed. If a linear IVF $\mathbf{L}: \mathcal{Y} \rightarrow I(\mathbb{R})$ is $gH$-continuous at the zero vector of $\mathcal{Y}$, then it is a bounded linear operator.
\end{lem}
The authors of \cite{Stefanini2019} provided the definition of $gH$-differentiability for IVFs using the midpoint-radius representation $\left[\frac{\overline{f}+\underline{f}}{2}, \frac{\overline{f}-\underline{f}}{2}\right]$ of an IVF $\textbf{F}$. However, as our main intention in this article is to illustrate all the things regarding IVF whether its lower boundary function $\underline{f}$ and upper boundary function $\overline{f}$ are readily available or not, we consider the Proposition $7$ of \cite{Stefanini2019} as the definition of $gH$-differentiability for IVFs, which is as follows.
\begin{dfn}\label{dghd} (\emph{$gH$-differentiability}).
Let $\mathcal{X}$ be a nonempty subset of $\mathbb{R}^n$. An IVF $\textbf{F}:\mathcal{X} \rightarrow I(\mathbb{R})$ is said to be $gH$-differentiable at a point $\bar{x} \in \mathcal{X}$ if there exists an IVF $\textbf{L}_{\bar{x}}(d)=d^T\odot\widehat{\textbf{A}}$, where $d\in \mathbb{R}^n$ and $\widehat{\textbf{A}}\in I(\mathbb{R})^n$, an IVF $\textbf{E}(\textbf{F}(\bar{x});d)$ and a $\delta~>~0$ such that
\[
\left(\textbf{F}(\bar{x}+d)\ominus_{gH} \textbf{F}(\bar{x})\right)=\textbf{L}_{\bar{x}}(d)\oplus\lVert d \rVert \odot \textbf{E}(\textbf{F}(\bar{x});d)~\text{for all}~d~\text{with}~\lVert d \rVert~<~ \delta,
\]
where $\textbf{E}(\textbf{F}(\bar{x});d)\rightarrow \textbf{0}$ as $\lVert d \rVert\rightarrow 0$.
\end{dfn}
\begin{thm}\emph{(See \cite{Stefanini2019})}.\label{td1}
Let an IVF $\textbf{F}$ on a nonempty subset $\mathcal{X}$ of $\mathbb{R}^n$ be $gH$-differentiable at $\bar{x}\in \mathcal{X}$. Then, for each $d = (d_1, d_2, \cdots, d_n)^T \in \mathbb{R}^n$, the $gH$-gradient of $\textbf{F}$ at $\bar{x}$ exists and the IVF $\textbf{L}_{\bar{x}}$ in Definition \ref{dghd} can be expressed by
\begin{equation}\label{edf3}
\textbf{L}_{\bar{x}}(d) = d^T \odot \nabla\textbf{F}(\bar{x})=\bigoplus_{i=1}^n d_i \odot D_i\textbf{F}(\bar{x}).
\end{equation}
\end{thm}
\begin{thm}\emph{(See \cite{Stefanini2019})}.\label{td}
Let an IVF $\textbf{F}$ on a nonempty subset $\mathcal{X}$ of $\mathbb{R}^n$ be $gH$-differentiable at $\bar{x}\in \mathcal{X}$. Then, $\textbf{F}$ has directional $gH$-derivative at $\bar{x}$ for every direction $d\in \mathbb{R}^n$ and
\[
\textbf{F}'(\bar{x})(d)= d^T \odot \nabla\textbf{F}(\bar{x})=\bigoplus_{i=1}^n d_i \odot D_i\textbf{F}(\bar{x})~\text{for all}~d\in \mathbb{R}^n.
\]
\end{thm}
\begin{thm}\label{td2}
Let an IVF $\textbf{F}$ on a nonempty open convex subset $\mathcal{X}$ of $\mathbb{R}^n$ be $gH$-differentiable at $x\in\mathcal{X}$. If the function $\textbf{F}$ is convex on $\mathcal{X}$, then
\[
(y-x)^T \odot \nabla\textbf{F}(x)~\preceq~ \textbf{F}(y)\ominus_{gH}\textbf{F}(x) ~ \text{ for all } x,~y\in \mathcal{X}.
\]
\end{thm}
\begin{proof}
Please see \ref{aptd2}.
\end{proof}
%
%
%
	

\section{Subdifferentiability of IVFs} \label{ssd}
\noindent Here we describe the concepts $gH$-subgradient and $gH$-subdifferential for convex IVFs and study their characteristics. In order to do this, we adopt the concept of subgradient for convex FVFs provided in \cite{Hai2018}.

\begin{dfn}\label{dsg}(\emph{gH-subgradient}). Let $\mathcal{X}$ be a nonempty convex subset of $\mathbb{R}^n$. An element $\widehat{\textbf{G}}=(\textbf{G}_1, \textbf{G}_2,\cdots,\textbf{G}_n)\in I(\mathbb{R})^n$ is said to be a $gH$-subgradient of a convex IVF $\textbf{F}: \mathcal{X} \rightarrow I(\mathbb{R})$ at $\bar{x}\in \mathcal{X}$ if
\begin{equation}\label{esg}
(x-\bar{x})^T\odot\widehat{\textbf{G}}\preceq \textbf{F}(x)\ominus_{gH} \textbf{F}(\bar{x}) \text{ for all } x\in \mathcal{X}.
\end{equation}
Due to Remark \ref{nlivf}, we can also define the $gH$-subgradient as $gH$-continuous linear IVF.\\
A $gH$-continuous linear IVF $\textbf{L}_{\bar{x}}:\mathcal{Y}\rightarrow I(\mathbb{R})$ is said to be $gH$-subgradient of $\textbf{F}$ at $\bar{x}\in \mathcal{X}$ if
\begin{equation}\label{esd}
\textbf{L}_{\bar{x}}(x-\bar{x})\preceq \textbf{F}(x)\ominus_{gH} \textbf{F}(\bar{x}) \text{ for all } x\in \mathcal{X},
\end{equation}
where $\mathcal{Y}$ is the smallest linear subspace of $\mathbb{R}^n$ containing $\mathcal{X}$.
\end{dfn}
\begin{dfn}\label{subdifferential}(\emph{gH-subdifferential}). The set $\partial \textbf{F}(\bar{x})$ of all $gH$-subgradients of the convex IVF $\textbf{F}: \mathcal{X}\subset\mathbb{R}^n \rightarrow I(\mathbb{R})$ at $\bar{x}\in \mathcal{X}$, where $\mathcal{X}$ is convex, is called $gH$-subdifferential of $\textbf{F}$ at $\bar{x}$.
\end{dfn}
Throughout this article, we express an element of $\partial\textbf{F}(\bar{x})$ either as as an element of $I(\mathbb{R})^n$ satisfying (\ref{esg}) or as an element of $\widehat{\mathcal{Y}}$ satisfying (\ref{esd}).

\begin{rmrk}\label{nsg1}
In view of Theorem \ref{td2}, it is to be noted that if $\textbf{F}$ is $gH$-differentiable at $\bar{x}\in \mathcal{X}$, then $\nabla\textbf{F}(\bar{x}) \in \partial \textbf{F}(\bar{x})$.
\end{rmrk}
\begin{example}\label{ex}
Let $\mathcal{X}$ be a nonempty convex subset of $\mathbb{R}$ and an IVF $\textbf{F}:\mathcal{X}\rightarrow I(\mathbb{R})$ be defined by $\textbf{F}(x) = \rvert x \rvert \odot \textbf{A} $, where $\textbf{0}\preceq\textbf{A}$. If $\textbf{G}\in I(\mathbb{R})$ is a $gH$-subgradient of $\textbf{F}$ at $\bar{x}=0$, then according to Definition \ref{dsg}, we have
 \[
 (x-\bar{x})\odot\textbf{G}\preceq \textbf{F}(x)\ominus_{gH} \textbf{F}(\bar{x}) \Longrightarrow \textbf{G}\odot x \preceq\textbf{A}\odot \rvert x \rvert.
\]
Therefore, for $x\leq0$, we have
\begin{equation}\label{hh}
\textbf{G}\odot x \preceq (-1)\odot\textbf{A}\odot x\Longrightarrow (-1)\odot\textbf{A}\preceq \textbf{G}
\end{equation}
and for $x\geq 0$, we have
\begin{equation}\label{h}
\textbf{G}\odot x \preceq \textbf{A}\odot x\Longrightarrow \textbf{G}\preceq \textbf{A}.
\end{equation}
With the help of (\ref{hh}) and (\ref{h}), we obtain
\[
(-1)\odot\textbf{A}\preceq \textbf{G}\preceq \textbf{A}.
\]
Hence, $\partial \textbf{F}(0)=\{\textbf{G}:(-1)\odot\textbf{A}\preceq \textbf{G}\preceq \textbf{A}\}$.\\

\noindent Considering $\textbf{A}=[1, 3]$, the IVF $\textbf{F}$ is depicted in Figure \ref{fsgm} by the shaded region within dashed lines, and two possible subgradients $\textbf{G}_1$, $\textbf{G}_2\in \partial\textbf{F}(0)$ of $\textbf{F}$ are illustrated by black and dark gray regions, respectively.

\begin{figure}[H]
\begin{center}
\includegraphics[scale=0.6]{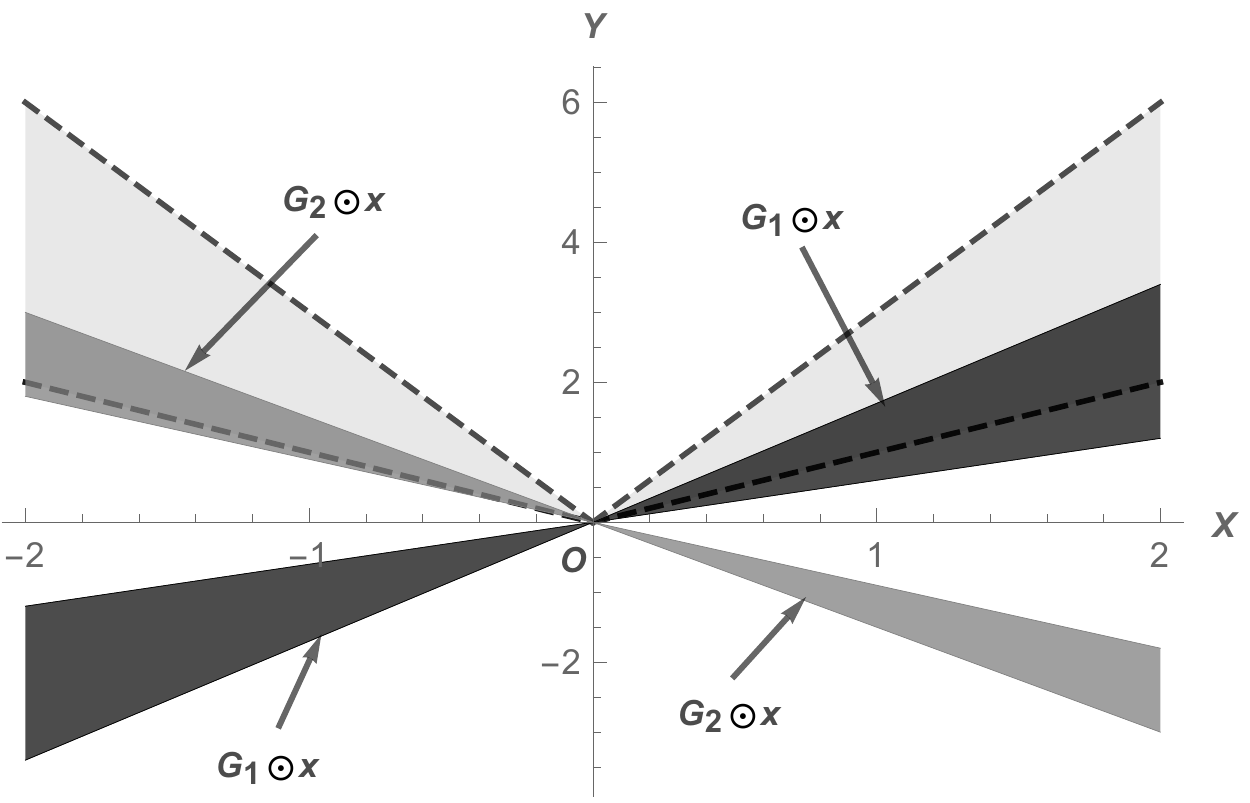}
    \caption{The IVF $\textbf{F}$ of Example \ref{ex} is depicted by the shaded region within dashed lines, and two possible subgradients $\textbf{G}_1$ and $\textbf{G}_2$ of $\textbf{F}$ are  illustrated by black and dark gray regions, respectively.}\label{fsgm}
\end{center}
\end{figure}
\end{example}
\begin{rmrk}\label{rsg}
It is noteworthy that
\begin{enumerate}[(i)]
    \item the author of \cite{Karaman2020} in Definition $2$ has proposed the concept of subgradient for IVFs by considering $L$ as linear real valued function. However, in Definition \ref{dsg} of the present article, we consider $\textbf{L}_{\bar x}$ as linear IVF. That's why our concept of subgradient in terms of linear function is more general.
    \item as IVFs are the special case of FVFs, one may think that we can adopt the concept of subgradient for FVFs of the article \cite{Zhang2005} as the concept of subgradient for IVFs. However, according to Definition $3.1$ of \cite{Zhang2005}, if we define the $gH$-subgradient $\widehat{\textbf{G}}$ satisfying the condition
    \begin{equation}\label{esga}
(x-\bar{x})^T\odot\widehat{\textbf{G}}\oplus\textbf{F}(\bar{x})\preceq \textbf{F}(x)
\end{equation}
instead of satisfying the condition \eqref{esg} in Definition \ref{dsg}, then Definition \ref{dsg} will be quite restrictive even for a $gH$-differentiable IVF. 
For instance, consider the following example.
\end{enumerate}
\end{rmrk}
\begin{example}
Let an IVF $\textbf{F}:[0, 2.5]\rightarrow I(\mathbb{R})$ be defined by
\begingroup\allowdisplaybreaks\begin{align*}
 \textbf{F}(x)&=[1, 1]\odot x^4\oplus [0, 1]\odot (x^2-x^4+34)\oplus[1, 6]\\
 &=[x^4+1, x^2+40]\\
 &=[\underline{f}(x),\overline{f}(x)]
\end{align*}\endgroup
Clearly, the real-valued functions $\underline{f}$ and $\overline{f}$ are differentiable at $\bar{x}=1$. Hence, the $gH$-derivative $\textbf{F}'(\bar{x})$ of $\textbf{F}$ at $\bar{x}=1$ exists due to Remark \ref{rd1}, and
\[
\nabla\textbf{F}(1)=\textbf{F}'(1)=[2, 4].
\]
Since
\[
\nabla\textbf{F}(1)\odot(2-1)=[2, 4]\preceq [3, 15]=\textbf{F}(2)\ominus_{gH}\textbf{F}(1)
\]
but
\[
\nabla\textbf{F}(1)\odot(2-1)\oplus\textbf{F}(1)=[4, 45]\npreceq[17, 44]=\textbf{F}(2)
\]
therefore, $\nabla\textbf{F}(1)\in \partial\textbf{F}(1)$ with respect to condition \eqref{esg} not respect to condition \eqref{esga}.
\end{example}
Now we provide an example of $gH$-subdifferential as a collection of $gH$-continuous linear IVF through Theorem \ref{tnf}. To do so, we introduce the concept of norm on the set $\widehat{\mathcal{Y}}$ of all $gH$-continuous linear IVFs on a linear subspace $\mathcal{Y}$ of $\mathbb{R}^n$.
\begin{dfn}(\emph{Norm on $\widehat{\mathcal{Y}}$}). A norm on the set $\widehat{\mathcal{Y}}$ of all $gH$-continuous linear IVF $\textbf{L}: \mathcal{Y} \rightarrow I(\mathbb{R})$ is defined by the function $\lVert \cdot \rVert_{\widehat{\mathcal{Y}}}: \widehat{\mathcal{Y}} \to \mathbb{R}_+$ such that
\[
\lVert \textbf{L}\rVert_{\widehat{\mathcal{Y}}} = \sup\limits_{x\neq 0}\frac{\lVert \textbf{L}(x)\rVert_{I(\mathbb{R})}}{\lVert x \rVert},~\text{where}~x\in\mathcal{Y}.
\]
To prove that the function ${\lVert \cdot \rVert}_{\widehat{\mathcal{Y}}}$ satisfies all the properties of a norm please see \ref{aplfnorm}.
\end{dfn}
\begin{lem}\label{lnf}
Let $\mathcal{Y}$ be a linear subspace of $\mathbb{R}^n$, and $\textbf{L}\in \widehat{\mathcal{Y}}$ be such that
\[
\textbf{L}(x)\preceq \textbf{C}\odot \lVert x \rVert~\text{for all}~x\in\mathcal{Y},
\]
where $\textbf{C}$ is a closed and bounded interval. Then,
\[
\lVert\textbf{L}(x)\rVert_{I(\mathbb{R})}\leq \lVert\textbf{C}\rVert_{I(\mathbb{R})} \lVert x \rVert~\text{for all}~x\in\mathcal{Y}.
\]
\end{lem}
\begin{proof}
Please see \ref{aplnf}.
\end{proof}
\begin{thm}\label{tnf}
Let $\mathcal{Y}$ be a linear subspace of $\mathbb{R}^n$ and $\textbf{F}: \mathcal{Y} \rightarrow I(\mathbb{R})$ be a convex IVF, defined by
\[
\textbf{F}(x)=\textbf{C}\odot \lVert x \rVert~\text{for all}~x\in\mathcal{Y},
\]
where $\textbf{C}\in I(\mathbb{R}_+)$. Then,
\[
\partial\textbf{F}(0)=\left\{\textbf{L}_{0}\in\widehat{\mathcal{Y}}\mid\lVert\textbf{L}_{0}\rVert_{\widehat{\mathcal{Y}}}\leq  \lVert\textbf{C}\rVert_{I(\mathbb{R})}\right\}.
\]
\end{thm}
\begin{proof}
Let $\textbf{L}_{0}\in\partial\textbf{F}(0)$. Therefore, for all nonzero $x\in\mathcal{Y}$,
\begingroup\allowdisplaybreaks\begin{align*}
&\textbf{L}_{0}(x-0)\preceq\textbf{F}(x)\ominus_{gH}\textbf{F}(0)\\
\Longrightarrow~&\textbf{L}_{0}(x)\preceq\textbf{C}\odot \lVert x \rVert\\
\Longrightarrow~&\lVert\textbf{L}_{0}(x)\rVert_{I(\mathbb{R})}\leq \lVert\textbf{C}\rVert_{I(\mathbb{R})} \lVert x \rVert,~\text{by Lemma \ref{lnf}}\\
\Longrightarrow~&\frac{\lVert\textbf{L}_{0}(x)\rVert_{I(\mathbb{R})}}{\lVert x \rVert}\leq \lVert\textbf{C}\rVert_{I(\mathbb{R})} \\
\Longrightarrow~&\sup_{x\neq 0}\frac{\lVert\textbf{L}_{0}(x)\rVert_{I(\mathbb{R})}}{\lVert x \rVert}\leq \lVert\textbf{C}\rVert_{I(\mathbb{R})}\\
\Longrightarrow~&\lVert\textbf{L}_{0}\rVert_{\widehat{\mathcal{Y}}}\leq  \lVert\textbf{C}\rVert_{I(\mathbb{R})}.
\end{align*}\endgroup
Hence,
\[
\partial\textbf{F}(0)=\left\{\textbf{L}_{0}\in\widehat{\mathcal{Y}}\mid\lVert\textbf{L}_{0}\rVert_{\widehat{\mathcal{Y}}}\leq  \lVert\textbf{C}\rVert_{I(\mathbb{R})}\right\}.
\]
\end{proof}
Next, we show that a $gH$-differentiable convex IVF has only one $gH$-subgradient, which is $gH$-gradient of the IVF. Also, we show that on a real linear subspace if the $gH$-subgradients of a convex IVF at a point exists, then the $gH$-directional derivative of the IVF at that point in each direction is maximum of all the products of $gH$-subgradients and the direction.

\begin{lem}\label{lddsg}
Let $\mathcal{X}$ be a nonempty convex subset of $\mathbb{R}^n$ and $\textbf{F}$ be a convex IVF on $\mathcal{X}$. Then, for an arbitrary $\bar{x}\in \mathbb{R}^n$
\[
\partial \textbf{F}(\bar{x})=\left\{\widehat{\textbf{G}} \in I(\mathbb{R})^n \mid h^T\odot\widehat{\textbf{G}}\preceq\textbf{F}'(\bar{x})(h) ~\text{for all}~h \in \mathcal{X}\right\},
\]
where $\textbf{F}'(\bar{x})(h)$ is $gH$-directional derivative of $\textbf{F}$ at $\bar{x}$ in the direction $h$.
\end{lem}
\begin{proof}
For an arbitrary $\widehat{\textbf{G}}\in \partial \textbf{F}(\bar{x})$, we have
\begin{align*}
(x-\bar{x})^T\odot\widehat{\textbf{G}}\preceq \textbf{F}(x)\ominus_{gH} \textbf{F}(\bar{x})~\text{for all}~~ x\in \mathcal{X}.
\end{align*}
Replacing $x$ by $\bar{x}+\lambda h$, where $\lambda >0$, we get
\[
(\lambda h)^T\odot\widehat{\textbf{G}}\preceq \textbf{F}(\bar{x}+\lambda h)\ominus_{gH} \textbf{F}(\bar{x}),
\]
which implies
\begin{align*}
&h^T\odot\widehat{\textbf{G}}\preceq \lim_{\lambda \to 0+}\frac{1}{\lambda}\odot\left(\textbf{F}(\bar{x}+\lambda h)\ominus_{gH}\textbf{F}(\bar{x})\right)\\
\Longrightarrow~& h^T\odot\widehat{\textbf{G}}\preceq\textbf{F}'(\bar{x})(h). \end{align*}
\end{proof}
\begin{thm}\label{tgd}
Let $\mathcal{X}$ be a nonempty subset of $\mathbb{R}^n$. If an IVF $\textbf{F}:\mathcal{X}\rightarrow I(\mathbb{R})$ is $gH$-differentiable at $\bar{x}\in\mathcal{X}$, then
\[
\partial \textbf{F}(\bar{x})=\left\{\nabla\textbf{F}(\bar{x})\right\}.
\]
\end{thm}
\begin{proof}
Let $\textbf{G}\in \partial \textbf{F}(\bar{x})$. Since $\textbf{F}$ is $gH$-differentiable at $\bar{x}$, in view of Theorem \ref{td} and Lemma \ref{lddsg}, we have
\begin{equation}\label{esg1}
h^T\odot\widehat{\textbf{G}}\preceq h^T\odot\nabla\textbf{F}(\bar{x})~\text{for all}~h \in \mathbb{R}^n.
\end{equation}
Replacing $h$ by $-h$ in the last relation we get
\[
(-h)^T\odot\widehat{\textbf{G}}\preceq (-h)^T\odot\nabla\textbf{F}(\bar{x}),
\]
which implies
\begin{equation}\label{esg2}
h^T\odot\nabla\textbf{F}(\bar{x})\preceq h^T\odot\widehat{\textbf{G}}~\text{for all}~h \in \mathbb{R}^n.
\end{equation}
Thus, the relations \eqref{esg1} and \eqref{esg2} together yield
\begin{equation}\label{esg3}
h^T\odot\nabla\textbf{F}(\bar{x})=h^T\odot\widehat{\textbf{G}}~\text{for all}~h \in \mathbb{R}^n.
\end{equation}
For each $i\in\{1, 2, \cdots, n\}$, by choosing $h=e_i$, we have
\[
D_i\textbf{F}(\bar{x})=\textbf{G}_i.
\]
Therefore,
\[
\nabla\textbf{F}(\bar{x})=\widehat{\textbf{G}}
\]
and hence,
\[
\partial \textbf{F}(\bar{x})=\left\{\nabla\textbf{F}(\bar{x})\right\}.
\]
\end{proof}
\begin{thm}\label{tddsg}
Let $\textbf{F}: \mathcal{Y} \rightarrow I(\mathbb{R})$  be a convex and $gH$-continuous IVF on a real linear subspace $\mathcal{Y}$ of $\mathbb{R}^n$. If $gH$-subdifferential $\partial \textbf{F}(\bar{x})$ of $\textbf{F}$ at an $\bar{x} \in \mathcal{Y}$ is nonempty, then
\[
\textbf{F}'(\bar{x})(h)=\max \left\{h^T\odot \widehat{\textbf{G}}\mid \widehat{\textbf{G}}\in\partial \textbf{F}(\bar{x})\right\}~\text{for all}~h \in \mathcal{Y},
\]
where $\textbf{F}'(\bar{x})(h)$ is $gH$-directional derivative of $\textbf{F}$ at $\bar{x}$ in the direction $h$.
\end{thm}
\begin{proof}
Let $gH$-subdifferential $\partial \textbf{F}(\bar{x})$ of $\textbf{F}$ at $\bar{x} \in \mathcal{Y}$ is nonempty. Since $\textbf{F}$ is convex on $\mathcal{Y}$, the $gH$-directional derivative of $\textbf{F}$ at $\bar{x}$ in every direction $h \in \mathcal{Y}$ exists due to Theorem $3.1$ in \cite{Ghosh2019derivative}. By Lemma $3.1$ in \cite{Ghosh2019derivative}, we have
\begingroup\allowdisplaybreaks\begin{align*}
&\frac{1}{\lambda}\odot\textbf{F}(\bar{x}+\lambda h)\ominus_{gH}\textbf{F}(\bar{x}) \preceq\textbf{F}(\bar{x}+h)\ominus_{gH}\textbf{F}(\bar{x}),~\text{where}~\lambda>0\\
\Longrightarrow~ & \textbf{F}'(\bar{x})(h)\preceq\textbf{F}(\bar{x}+h)\ominus_{gH}\textbf{F}(\bar{x})
\end{align*}\endgroup
for all $h \in \mathcal{Y}$. Hence, in view of Lemma \ref{lddsg}, we obtain
\[
\textbf{F}'(\bar{x})(h)=\max \left\{\widehat{\textbf{G}}^T\odot h\mid \widehat{\textbf{G}}\in\partial \textbf{F}(\bar{x})\right\}~\text{for all}~h \in \mathcal{Y}.
\]
\end{proof}
Next, we show that $gH$-subdifferentials of a convex IVF are bounded and closed. To do so, we define a mapping $\mathcal{W}: I(\mathbb{R})^n \rightarrow \mathbb{R}^n$ by
\begin{equation}\label{wmap}
\mathcal{W}(\widehat{\textbf{A}})=\mathcal{W}(\textbf{A}_1, \textbf{A}_2, \cdots, \textbf{A}_n)=(w\underline{a}_1+w'\overline{a}_1, w\underline{a}_2+w'\overline{a}_2, \cdots, w\underline{a}_n+w'\overline{a}_n)^T,
\end{equation}
where $w$, $w'\in [0, 1]$ with $w+w'=1$.
\begin{lem}\label{lsgbd1}
For any $\widehat{\textbf{A}}\in I(\mathbb{R})^n$ and $d\in \mathbb{R}^n$,
\[
d^T\odot\widehat{\textbf{A}}\preceq [c, c]\Longrightarrow d^T\mathcal{W}\left(\widehat{\textbf{A}}\right)\leq 2c,
\]
where the map $\mathcal{W}$ is defined by (\ref{wmap}).
\end{lem}
\begin{proof}
Please see \ref{aplsgbd1}.
\end{proof}
\begin{lem}\label{lsgbd2}
For any $\widehat{\textbf{A}}\in I(\mathbb{R})^n$,
\[
\lVert\mathcal{W}\left(\widehat{\textbf{A}}\right)\rVert~ \text{is finite}\Longrightarrow {\lVert\widehat{\textbf{A}}\rVert}_{I(\mathbb{R})^n}~ \text{is finite},
\]
where the map $\mathcal{W}$ is defined by (\ref{wmap}).
\end{lem}
\begin{proof}
Please see \ref{aplsgbd2}.
\end{proof}
\begin{thm}\label{tbd}
Let $\mathcal{X}$ be a compact convex subset of $\mathbb{R}^n$ and $\textbf{F}:\mathcal{X} \rightarrow \mathbb{R}$ be a convex IVF. Then, the set  $\bigcup\limits_{x\in\mathcal{X}}\partial\textbf{F}(x)$ is bounded.
\end{thm}
\begin{proof}
We claim that the set  $\bigcup\limits_{x\in\mathcal{X}}\partial\textbf{F}(x)$ is bounded. On contrary, there exists a sequence $\{x_k\}$ on $\mathcal{X}$ and an unbounded sequence $\{\widehat{\textbf{G}}_k\}$, where $\widehat{\textbf{G}}_k \in \partial\textbf{F}(x_k)$, such that
\[
0<\lVert\widehat{\textbf{G}}_k\rVert< \lVert\widehat{\textbf{G}}_{k+1}\rVert, ~k\in\mathbb{N}.
\]
Let us take $d_k=\frac{\mathcal{W}\left(\widehat{\textbf{G}}_k\right)} {\lVert\mathcal{W}\left(\widehat{\textbf{G}}_k\right)\rVert}$, where the mapping $\mathcal{W}$ is defined by (\ref{wmap}). By Definition \ref{dsg} we have
\begingroup\allowdisplaybreaks\begin{align*}
&d_k^T\odot\widehat{\textbf{G}}_k\preceq\textbf{F}(x_k+d_k)\ominus_{gH} \textbf{F}(x_k) = \max\left\{\underline{f}(x_k+d_k)-\underline{f}(x_k), ~\overline{f}(x_k+d_k)-\overline{f}(x_k)\right\} \\
\Longrightarrow~ & d_k^T\odot\widehat{\textbf{G}}_k\preceq [c, c],~\text{where}~ \max\left\{\overline{f}(x)\right | x \in \mathcal{X}\} \le c\\
\Longrightarrow~ & d_k^T\mathcal{W}\left(\widehat{\textbf{G}}_k\right)\leq 2c,~\text{by Lemma \ref{lsgbd1}}\\
\Longrightarrow~ & \lVert\mathcal{W}\left(\widehat{\textbf{G}}_k\right)\rVert\leq 2c.
\end{align*}\endgroup
Since $\textbf{F}$ is convex on $\mathbb{R}^n$, in view of Theorem \ref{tc} and Lemma \ref{lc2}, $\underline{f}$ and $\overline{f}$ are continuous on $\mathbb{R}^n$. As $\{x_k\}$ and $\{d_k\}$ are bounded and the real-valued functions $\underline{f}$ and $\overline{f}$ are continuous, by the property of real-valued function, $c$ is finite. Thus, $\lVert\mathcal{W}\left(\widehat{\textbf{G}}_k\right)\rVert$ is finite and hence, due to Lemma \ref{lsgbd2}, $\lVert\widehat{\textbf{G}}_k\rVert_{I(\mathbb{R}^n)}$ is finite. Therefore, the sequence $\{\widehat{\textbf{G}}_k\}$ is bounded, which is a contradiction. Hence, the set $\bigcup\limits_{x\in\mathcal{X}}\partial\textbf{F}(x)$ is bounded.
\end{proof}
\begin{thm}\label{tcd}
Let $\mathcal{X}$ be a nonempty convex subset of $\mathbb{R}^n$ and $\textbf{F}$ be a convex IVF on $\mathcal{X}$. Then, for every $\bar{x}\in\mathcal{X}$, $ \partial \textbf{F}(\bar{x})$ is closed.
\end{thm}

\begin{proof}
Let $\left\{\widehat{\textbf{G}}_{k}\right\}$ be an arbitrary sequence in $\partial \textbf{F}(x)$ which converges to $\widehat{\textbf{G}}\in I(\mathbb{R})^n$, where $\widehat{\textbf{G}}_k=\left(\textbf{G}_{k1}, \textbf{G}_{k2}, \cdots, \textbf{G}_{kn}\right)$ and $\widehat{\textbf{G}}=\left(\textbf{G}_{1}, \textbf{G}_{2}, \cdots, \textbf{G}_{n}\right)$.\\

\noindent Since, $\widehat{\textbf{G}}_k\in \partial \textbf{F}(\bar{x})$, for all $d\in\mathcal{X}$ we have
\[
d^T\odot \widehat{\textbf{G}}_k\preceq \textbf{F}(\bar{x}+d)\ominus_{gH} \textbf{F}(\bar{x}),
\]
i.e.,
\begin{equation}\label{ec1}
\bigoplus_{i=1}^n d_i\odot \textbf{G}_{ki}\preceq \textbf{F}(\bar{x}+d)\ominus_{gH} \textbf{F}(\bar{x}).
\end{equation}
Due to Remark \ref{ria1}, without loss of generality, let the first $p$ components of $d$ be nonnegative and the rest $n-p$ components be negative. Therefore, from  (\ref{ec1}), we get
\begingroup\allowdisplaybreaks\begin{align*}
&\bigoplus_{i=1}^p d_i\odot \textbf{G}_{ki}\oplus\bigoplus_{j=p+1}^n d_j\odot \textbf{G}_{kj}\preceq \textbf{F}(\bar{x}+d)\ominus_{gH} \textbf{F}(\bar{x})\\
\Longrightarrow~ & \bigoplus_{i=1}^p \left[\underline{g_{ki}}d_i, \overline{g_{ki}}d_i\right]\oplus\bigoplus_{j=p+1}^n \left[\overline{g_{kj}}d_j, \underline{g_{kj}}d_j\right]\preceq \textbf{F}(\bar{x}+d)\ominus_{gH} \textbf{F}(\bar{x})\\
\Longrightarrow~ & \left[\sum_{i=1}^p\underline{g_{ki}}d_i+\sum_{j=p+1}^n\overline{g_{kj}}d_j, \sum_{i=1}^p\overline{g_{ki}}d_i+\sum_{j=p+1}^n\underline{g_{kj}}d_j\right]\preceq \textbf{F}(\bar{x}+d)\ominus_{gH} \textbf{F}(\bar{x}).
\end{align*}\endgroup
Therefore, we get
\begin{equation}\label{ec2}
\sum_{i=1}^p\underline{g_{ki}}d_i+\sum_{j=p+1}^n\overline{g_{kj}}d_j\leq \min\left\{\underline{f}(\bar{x}+h)-\underline{f}(\bar{x}),\overline{f}(\bar{x}+h)-\overline{f}(\bar{x})\right\}
\end{equation}
\begin{equation}\label{ec3}
\sum_{i=1}^p\overline{g_{ki}}d_i+\sum_{j=p+1}^n\underline{g_{kj}}d_j\leq \max\left\{\underline{f}(\bar{x}+h)-\underline{f}(\bar{x}),\overline{f}(\bar{x}+h)-\overline{f}(\bar{x})\right\}
\end{equation}
Since the sequence $\left\{\widehat{\textbf{G}}_{k}\right\}$ converges to $\widehat{\textbf{G}}$, in view of Remark \ref{ncs}, the sequences $\left\{\underline{g_{ki}}\right\}$ and $\left\{\overline{g_{ki}}\right\}$ converge to $\underline{g_i}$ and $\overline{g_i}$, respectively, for all $i$. Thus, by (\ref{ec2}) and (\ref{ec3}), we have

\begingroup\allowdisplaybreaks\begin{align*}
\left(\sum_{i=1}^p\underline{g_{ki}}d_i+\sum_{j=p+1}^n\overline{g_{kj}}d_j\right) \to \left(\sum_{i=1}^p\underline{g_i}d_i+\sum_{j=p+1}^n\overline{g_j}d_j\right) \leq \min & \Big\{\underline{f}(\bar{x}+h)-\underline{f}(\bar{x}),\\
&~~~~\overline{f}(\bar{x}+h)-\overline{f}(\bar{x})\Big\}
\end{align*}\endgroup
and

\begingroup\allowdisplaybreaks\begin{align*}
\left(\sum_{i=1}^p\overline{g_{ki}}d_i+\sum_{j=p+1}^n\underline{g_{kj}}d_j\right)
\to
\left(\sum_{i=1}^p\overline{g_i}d_i+\sum_{j=p+1}^n\underline{g_j}d_j\right) \leq \max & \Big\{\underline{f}(\bar{x}+h)-\underline{f}(\bar{x}),\\
&~~~~\overline{f}(\bar{x}+h)-\overline{f}(\bar{x})\Big\}.
\end{align*}\endgroup
Hence,
\begingroup\allowdisplaybreaks\begin{align*}
&\left[\sum_{i=1}^p\underline{g_i}d_i+\sum_{j=p+1}^n\overline{g_j}d_j, \sum_{i=1}^p\overline{g_i}d_i+\sum_{j=p+1}^n\underline{g_j}d_j\right]\preceq \textbf{F}(\bar{x}+d)\ominus_{gH} \textbf{F}(\bar{x})\\
\Longrightarrow~ & \bigoplus_{i=1}^p \left[\underline{g_i}d_i, \overline{g_i}d_i\right]\oplus\bigoplus_{j=p+1}^n \left[\overline{g_j}d_j, \underline{g_j}d_j\right]\preceq \textbf{F}(\bar{x}+d)\ominus_{gH} \textbf{F}(\bar{x})\\
\Longrightarrow~ &\bigoplus_{i=1}^p d_i\odot \textbf{G}_{i}\oplus\bigoplus_{j=p+1}^n d_j\odot \textbf{G}_{j}\preceq \textbf{F}(\bar{x}+d)\ominus_{gH} \textbf{F}(\bar{x})\\
\Longrightarrow~ & d^T\odot \widehat{\textbf{G}}\preceq \textbf{F}(\bar{x}+d)\ominus_{gH} \textbf{F}(\bar{x})
\end{align*}\endgroup
for all $d\in\mathcal{X}$. Therefore, $\widehat{\textbf{G}} \in \partial \textbf{F}(\bar{x})$ and hence, $ \partial \textbf{F}(\bar{x})$ is closed.
\end{proof}
In the following theorem, we prove that if a convex IVF has $gH$-subgradients in all over its domain, then the IVF is $gH$-Lipschitz continuous on its domain.

\begin{lem}\label{llc1}
For any $x \in \mathbb{R}^n$ and $\widehat{\textbf{A}}=\left(\textbf{A}_1, \textbf{A}_2, \cdots, \textbf{A}_n\right)\in I(\mathbb{R})^n$,
\[
x^T\odot\widehat{\textbf{A}}\preceq \lVert x \rVert\odot\left[\lVert \widehat{\textbf{A}} \rVert _{I(\mathbb{R})^n}, \lVert \widehat{\textbf{A}} \rVert _{I(\mathbb{R})^n}\right].
\]
\end{lem}
\begin{proof}
Please see \ref{apllc1}.
\end{proof}
\begin{lem}\label{llc2}
Let $\mathcal{X}$ be a nonempty subset of $\mathbb{R}^n$ and $\textbf{F}$ be an IVF on $\mathcal{X}$ such that
\[
\textbf{F}(x)\ominus_{gH}\textbf{F}(y)\preceq \textbf{C} \odot \lVert x-y \rVert~\text{for all}~ x,~y\in\mathcal{X},
\]
where $\textbf{C}=[\underline{c}, \overline{c}]=[c, c]$. Then,
\[
\lVert \textbf{F}(x)\ominus_{gH}\textbf{F}(y)\rVert_{I(\mathbb{R})}\leq c\lVert x-y \rVert~\text{for all}~ x,~y\in \mathcal{X}.
\]
\end{lem}
\begin{proof}
Please see \ref{apllc2}.
\end{proof}
\begin{thm}\label{tlc}
Let $\mathcal{X}$ be a nonempty compact convex subset of $\mathbb{R}^n$ and $\textbf{F}$ be a convex IVF on $\mathcal{X}$ such that $\textbf{F}$ has $gH$-subgradient at every $x\in\mathcal{X}$. Then, $\textbf{F}$ is $gH$-Lipschitz continuous on $\mathcal{X}$.
\end{thm}
\begin{proof}
Since $\textbf{F}$ has $gH$-subgradient at every $x\in\mathcal{X}$, there exists a $\widehat{\textbf{G}}\in I(\mathbb{R})^n$ such that
\begingroup\allowdisplaybreaks\begin{align*}\label{gh}
&(y-x)^T\odot\widehat{\textbf{G}}\preceq\textbf{F}(y)\ominus_{gH} \textbf{F}(x) \\
\Longrightarrow~ &(-1)\odot \left((x-y)^T\odot\widehat{\textbf{G}}\right)\preceq\textbf{F}(y)\ominus_{gH} \textbf{F}(x) \\
\Longrightarrow~ &\textbf{F}(x)\ominus_{gH} \textbf{F}(y)\preceq (x-y)^T\odot\widehat{\textbf{G}} \\
\Longrightarrow~ &\textbf{F}(x)\ominus_{gH} \textbf{F}(y)\preceq \lVert x-y \rVert\odot\left[\lVert \widehat{\textbf{G}} \rVert _{I(\mathbb{R})^n}, \lVert \widehat{\textbf{G}} \rVert _{I(\mathbb{R})^n}\right], ~\text{by Lemma \ref{llc1}}\\
\Longrightarrow~ &\lVert\textbf{F}(x)\ominus_{gH} \textbf{F}(y)\rVert_{I(\mathbb{R})}\leq \lVert \widehat{\textbf{G}} \rVert _{I(\mathbb{R})^n}\lVert x-y \rVert, ~\text{by Lemma \ref{llc2}.}
\end{align*}\endgroup
Considering $L=\sup\limits_{\widehat{\textbf{G}}\in \bigcup\limits_{x\in\mathcal{X}}\partial\emph{\textbf{F}}(x)}\lVert \widehat{\textbf{G}}\rVert_{I(\mathbb{R})^n}$, we have
\[
\lVert\textbf{F}(x)\ominus_{gH} \textbf{F}(y)\rVert_{I(\mathbb{R})}\leq L\lVert x-y \rVert ~\text{for all}~x,~y\in \mathcal{X}.
\]
Hence, $\textbf{F}$ is $gH$-Lipschitz continuous on $\mathcal{X}$.
\end{proof}
Now we show another two important characteristics of $gH$-subdifferential of a convex IVF.
\begin{thm}\emph{(Chain rule)}.\label{tcr}
Let $\mathcal{X}$ be a nonempty convex subset of $\mathbb{R}^n$ and an IVF $\textbf{F}$ be defined by
\[\textbf{F}(x)=\textbf{H}(Ax)~\text{for all}~ x\in\mathcal{X},
\]
where $\textbf{H}:\mathbb{R}^m \rightarrow I(\mathbb{R})$ is a convex IVF  and $A$ is an $m\times n$ matrix with real entries.
Then,
\[
\partial \textbf{F}(x)=\{A^T\odot\widehat{\textbf{G}}_m \mid \widehat{\textbf{G}}_m \in \partial\textbf{H}(Ax),~\text{where}~\widehat{\textbf{G}}_m \in I(\mathbb{R})^m~\text{and}~x\in\mathcal{X}\}.
\]
\end{thm}
\begin{proof}
By the definition of $gH$-subdifferentiability of $\textbf{H}$ at $A(x)$, for any $x\in\mathcal{X}$, we have a $\widehat{\textbf{G}}_m \in I(\mathbb{R})^m$ such that
\[
(Ay-Ax)^T \odot \widehat{\textbf{G}}_m\preceq \textbf{H}(Ay)\ominus_{gH} \textbf{H}(Ax) ~\text{for all}~ y\in \mathcal{X},
\]
which implies
\begingroup\allowdisplaybreaks\begin{align*}
&(A(y-x))^T \odot \widehat{\textbf{G}}_m\preceq \textbf{H}(Ay)\ominus_{gH} \textbf{H}(Ax) \nonumber \\
\Longrightarrow~& (y-x)^T \odot (A^T \odot\widehat{\textbf{G}}_m)\preceq \textbf{H}(Ay)\ominus_{gH} \textbf{H}(Ax) \nonumber \\
\Longrightarrow~& (y-x)^T \odot (A^T \odot\widehat{\textbf{G}}_m)\preceq \textbf{F}(y) \ominus_{gH} \textbf{F}(x).
\end{align*}\endgroup
Since $(A^T \odot\widehat{\textbf{G}}_m)\in I(\mathbb{R})^n$, by Definition \ref{dsg},
\[
\partial\textbf{F}(x)=\left\{A^T\odot\widehat{\textbf{G}}_m \mid \widehat{\textbf{G}}_m \in \partial\textbf{H}(Ax),~\text{where}~\widehat{\textbf{G}}_m \in I(\mathbb{R})^m~\text{and}~x\in\mathcal{X}\right\}.
\]
\end{proof}
\begin{thm}\emph{($gH$-subdifferential of a sum)}.\label{tss}
Let $\mathcal{X}$ be a nonempty convex subset of $\mathbb{R}^n$ and an IVF $\textbf{F}$ be defined by
\[
\textbf{F}(x)=\bigoplus_{i=1}^m\textbf{F}_i(x)~\text{for all}~ x\in\mathcal{X},
\]
where each $\textbf{F}_i:\mathcal{X}\rightarrow I(\mathbb{R})$ is a convex IVF. Then,
\[
\partial \textbf{F}(x)=\bigoplus_{i=1}^m\partial\textbf{F}_i(x)~\text{for all}~ x\in\mathcal{X}.
\]
\end{thm}
\begin{proof}
We write
\[
\textbf{F}(x)=\textbf{H}(Ax)~\text{for all}~ x\in\mathcal{X},
\]
where $A$ is a matrix, defined by $Ax=(x, x, \cdots, x)^T$ for all $x\in\mathcal{X}$
and $\textbf{H}:\mathbb{R}^{mn} \rightarrow I(\mathbb{R})$ is an IVF, defined by
\[
\textbf{H}(y)=\textbf{H}(y_1, y_2, \cdots, y_m)=\bigoplus_{i=1}^m\textbf{F}_i(y_i)~\text{for all}~ y\in \mathbb{R}^{mn}.
\]
Thus, by Theorem \ref{tcr}, we have
\[
\partial \textbf{F}(x)=\bigoplus_{i=1}^m\partial\textbf{F}_i(x)~\text{for all}~ x\in\mathcal{X}.
\]
\end{proof}


\section{Convex IOP and its Optimality Conditions}\label{siop}

\noindent In this section, we explore the relation of efficient solutions to the following IOP:
\begin{equation}\label{IOP}
\displaystyle \min_{x \in \mathcal{X}} \textbf{F}(x),
\end{equation}
where $\textbf{F}$ is a convex IVF on the nonempty convex subset $\mathcal{X}$ of $\mathbb{R}^n$, with the $gH$-subgradients of $\textbf{F}$. The IOP with convex IVFs is known as convex IOP.\\

\noindent The concept of an efficient solution of the IOP (\ref{IOP}) is defined below.
\begin{dfn}
(\emph{Efficient solution} \cite{Ghosh2019derivative}). A point $\bar{x} \in \mathcal{X}$ is called an efficient solution of the IOP (\ref{IOP}) if $\textbf{F}(x) \nprec \textbf{F} (\bar{x})$ for all $x (\neq \bar{x}) \in
\mathcal{X}$.
\end{dfn}
%
%
%
%
%
%
\begin{thm}\emph{(Optimality condition)}.\label{tsoc1}
Let $\mathcal{X}$ be a nonempty convex subset of $\mathbb{R}^n$ and $\textbf{F}: \mathcal{X} \rightarrow I(\mathbb{R})$ be a convex IVF. If $\widehat{\textbf{0}}\in \partial \textbf{F}(\bar{x})$ for some $\bar{x} \in \mathcal{X}$, where $\widehat{\textbf{0}} = (\textbf{0}, \textbf{0}, \cdots , \textbf{0})$, then $\bar{x}$ is an efficient solution of the IOP (\ref{IOP}).
\end{thm}
\begin{proof}
Let $ \widehat{\textbf{0}}\in \partial \textbf{F}(\bar{x})$. Then for all $x\in \mathcal{X}$,
\begingroup\allowdisplaybreaks\begin{align*}
(x-\bar{x})^T\odot\widehat{\textbf{0}}\preceq\textbf{F}(x)\ominus_{gH} \textbf{F}(\bar{x})~\Longrightarrow~& \textbf{0}\preceq\textbf{F}(x)\ominus_{gH} \textbf{F}(\bar{x})\\
\Longrightarrow~& \textbf{F}(\bar{x})\preceq \textbf{F}(x)\\
\Longrightarrow~& \textbf{F}(x)\nprec \textbf{F}(\bar{x}).
\end{align*}\endgroup
Hence, $\bar{x}$ is an efficient solution of the IOP (\ref{IOP}).
\end{proof}
For instance, consider the following example.
\begin{example}\label{exsgm}
Let us consider the following IOP:
\begin{equation}\label{exsgmiop}
\min_{x\in \mathcal{X}=[-2, 6]} \textbf{F}(x)=
\begin{cases}
[-2, 5] \ominus_{gH} [-1, 0]\odot \lvert x-2 \rvert, & \text{if } 1 \leq x \leq 3 \\
[-2, 3] \oplus [1, 2]\odot \lvert x-2 \rvert,  & \text{otherwise}.
\end{cases}\end{equation}
Let $\bar{x}=2$, then we have
\[
\textbf{F}(x)\ominus_{gH} \textbf{F}(\bar{x})=
\begin{cases}
[0, \lvert x-2 \rvert], & \text{for } 1 \leq x \leq 3 \\
[2\lvert x-2 \rvert-2, \lvert x-2 \rvert], & \text{for } 0 \leq x \leq 1 \text{ and } 3 \leq x \leq 4 \\
[\lvert x-2 \rvert, 2\lvert x-2 \rvert-2],  & \text{otherwise}.
\end{cases}
\]
Thus,
\[
\textbf{0}=(x-\bar{x})\odot\textbf{0}\preceq\textbf{F}(x)\ominus_{gH} \textbf{F}(\bar{x})~\text{for all}~ x\in \mathcal{X}.
\]
Therefore, $\textbf{0}\in\partial\textbf{F}(\bar{x})$.\\
\begin{figure}[H]
\begin{center}
\includegraphics[scale=0.6]{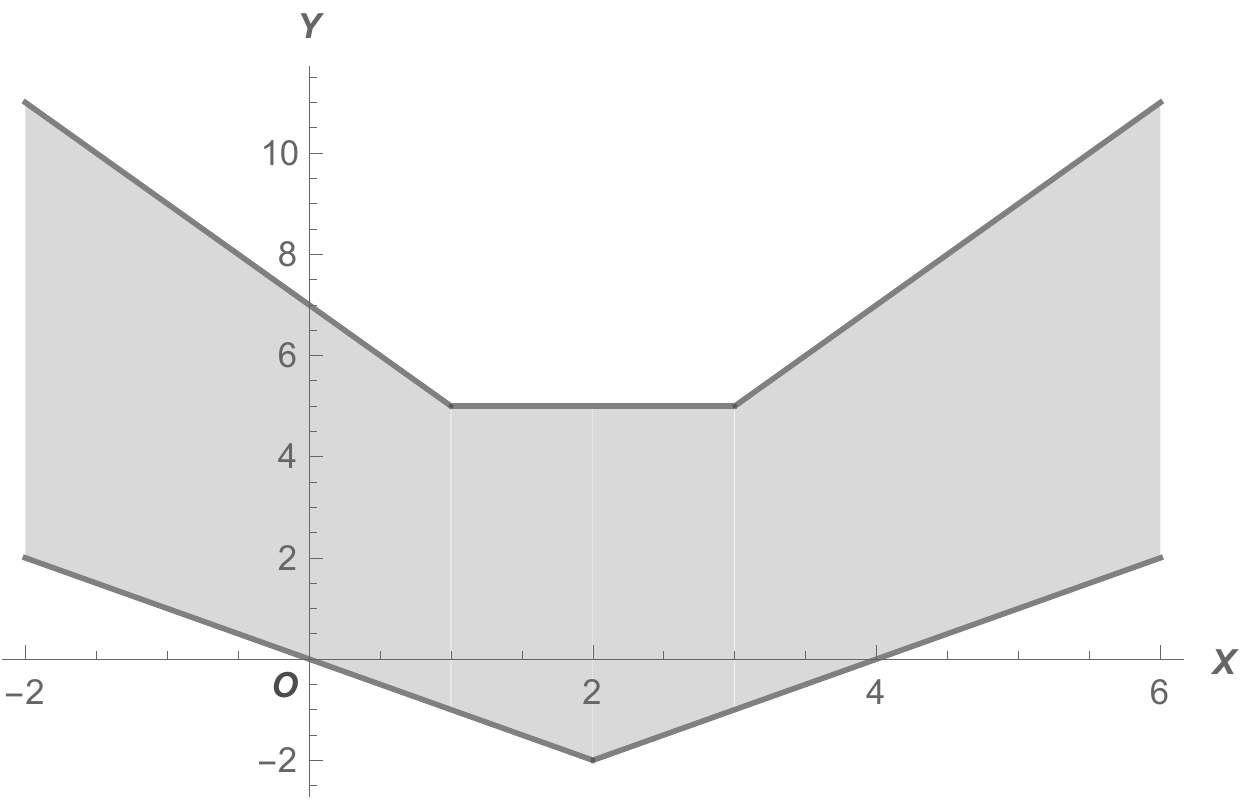}
    \caption{IVF $\textbf{F}$ of the IOP (\ref{exsgmiop}) is illustrated by gray shaded region.}\label{foptsgm}
\end{center}
\end{figure}
\end{example}
The graph of the IVF $\textbf{F}$ is depicted by the gray shaded region in Figure \ref{foptsgm}. From Figure \ref{foptsgm}, it is to be observed that there does not exist any $x(\neq \bar{x}) \in \mathcal{X}$ such that $\textbf{F}(x)\prec\textbf{F}(\bar{x})=[-2, 5]$. Hence, $\bar{x}=2$ is the efficient solution of the IOP (\ref{exsgmiop}).
The following example shows that the converse of Theorem \ref{tsoc1} is not true.
\begin{example}\label{nsg2}
Consider the following IOP:
\begin{equation}\label{ex1iop}
\min_{x\in \mathcal{X}} \textbf{F}(x)=[1, 2] \odot x^2 \ominus [0, 2] \odot (x+1) \oplus [4, 6],
\end{equation}
where $\mathcal{X}=[-1, 2]$.\\

\noindent Since $\underline{f}(x) = x^2-2x+2$ and $\overline{f}(x) =2x^2+6$ are convex on $\mathcal{X}$, the IVF $\textbf{F}$ is convex on $\mathcal{X}$ by Lemma \ref{lc1}. Further, as $\underline{f}$ and $\overline{f}$ are differentiable in $\mathcal{X}$, the IVF $\textbf{F}$ is $gH$-differentiable in $\mathcal{X}$ by Remark \ref{rd1}. Hence,
\[
\partial \textbf{F}(x)=\left\{\nabla \textbf{F}(x)\right\}=\left\{[2, 4]\odot x\ominus [0, 2]\right\}~\text{for all}~ x\in \mathcal{X}.
\]
The graph of the IVF $\textbf{F}$ is illustrated by the gray shaded region in Figure \ref{fsgm1}. From Figure \ref{fsgm1}, It is clear that for any $\bar{x} \in [0, 1]$, there does not exist any $x(\neq \bar{x}) \in \mathcal{X}$ such that $\textbf{F}(x)\prec\textbf{F}(\bar{x})$. Therefore, each $\bar{x} \in [0, 1]$ is an efficient solution of the IOP (\ref{ex1iop}). The region of the efficient solutions of the IOP (\ref{ex1iop}) is marked by bold black line on the $x$-axis in Figure \ref{fsgm1}. However, for each $x \in [0, 1]$,
\[
\nabla \textbf{F}(x)=[2x-2, 4x]\neq \textbf{0}
\]
and hence, $\textbf{0}\not\in\partial\textbf{F}(x)$.

\begin{figure}[H]
\begin{center}
\includegraphics[scale=0.6]{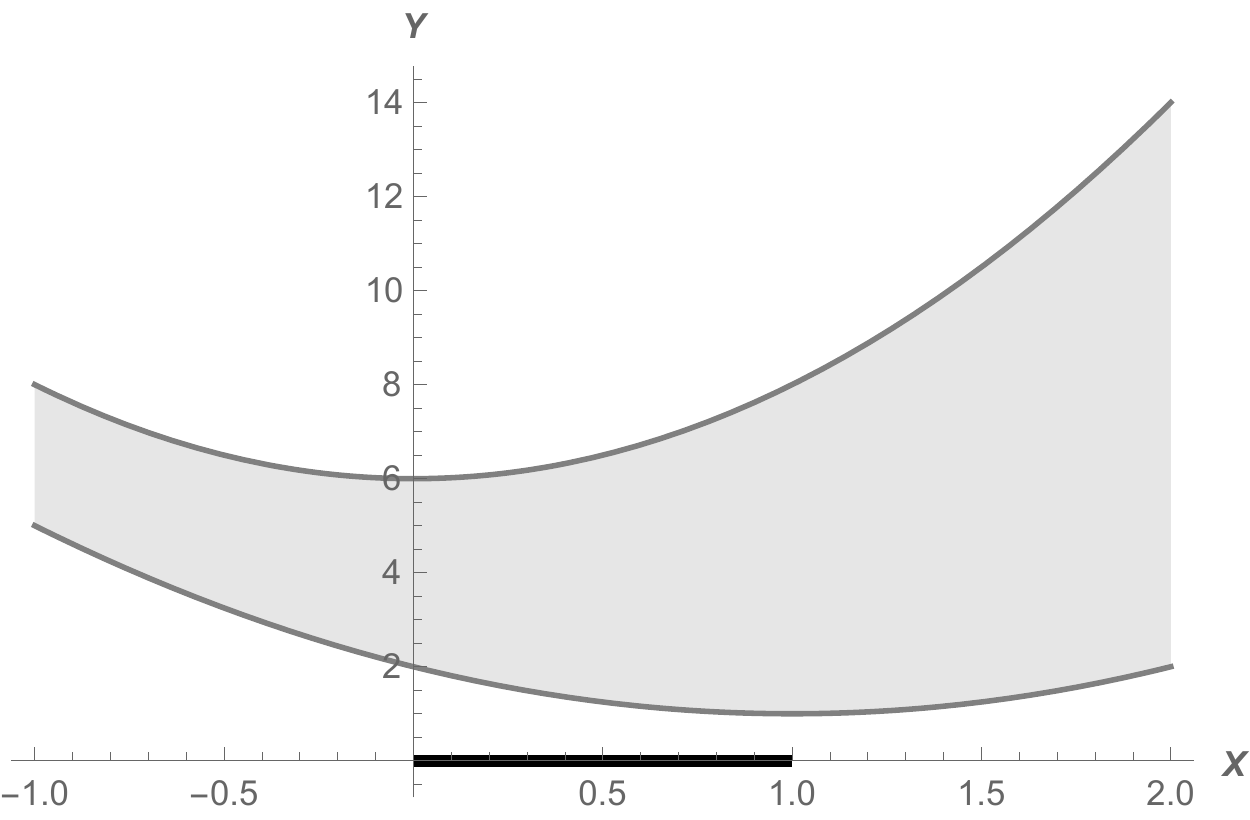}
    \caption{IVF $\textbf{F}$ and efficient solutions of the IOP (\ref{ex1iop}) are depicted by gray shaded region and bold black line on $x$-axis, respectively.}\label{fsgm1}
\end{center}
\end{figure}
\end{example}
\begin{thm}\emph{(Optimality condition)}.\label{tsoc2}
Let $\mathcal{X}$ be a nonempty convex subset of $\mathbb{R}^n$ and $\textbf{F}: \mathcal{X} \rightarrow I(\mathbb{R})$ be a convex IVF. If there exists a $\widehat{\textbf{G}}\in \partial \textbf{F}(\bar{x})$ for some $\bar{x} \in \mathcal{X}$, such that
\begin{equation}\label{esgo}
(x-\bar{x})^T\odot\widehat{\textbf{G}}\nprec \textbf{0}~\text{for all}~x\in\mathcal{X},
\end{equation}
then $\bar{x}$ is an efficient solution of the IOP (\ref{IOP}).
\end{thm}
\begin{proof}
Let there exists a $\widehat{\textbf{G}} \in \partial \textbf{F}(\bar{x})$ for which the relation (\ref{esgo}) is true. Then, by Definition \ref{dsg} of $gH$-subgradient and the relation (\ref{esgo}), we obtain
\begin{align*}
&\textbf{F}(x)\ominus_{gH} \textbf{F}(\bar{x})\nprec \textbf{0}\\
\Longrightarrow~& \textbf{F}(x)\nprec \textbf{F}(\bar{x})
\end{align*}
for all $x\in \mathcal{X}$. Hence, $\bar{x}$ is an efficient solution of the IOP (\ref{IOP}).
\end{proof}
\begin{rmrk}
The converse of Theorem \ref{tsoc2} is not true. For example, consider the IOP (\ref{ex1iop}) of Example \ref{nsg2}. We have seen that each point $\bar{x}\in[0, 1]$ is an efficient solution of the IOP (\ref{ex1iop}). However, at $\bar{x}=0$,
\[
(x-\bar{x})\odot\widehat{\textbf{G}}=(x-\bar{x})\odot\nabla\textbf{F}(\bar{x})=[-2, 0]\odot x\prec \textbf{0}~\text{for all}~x\in (0, 2]\subset\mathcal{X}.
\]
\end{rmrk}
%
%
%

\section{Conclusion and Future Directions} \label{scf}

\noindent In this article, the concepts of $gH$-subgradients and $gH$-subdifferentials of convex IVFs with their several important characteristics have been provided. It has been shown that the $gH$-subdifferential of a convex IVF is closed (Theorem \ref{tbd}) and convex (Theorem \ref{tcd}); the $gH$-subdifferential of a $gH$-differentiable convex IVF contains only $gH$-gradient of that IVF (Theorem \ref{tgd}). It has been observed that on a real linear subspace if the $gH$-subgradients of a convex IVF at a point exists, then the $gH$-directional derivative of the IVF at that point in each direction is maximum of all the products of $gH$-subgradients and the direction (Theorem \ref{tddsg}). Also, it has been observed that a convex IVF is $gH$-Lipschitz continuous if it has $gH$-subgradient at each point in its domain (Theorem \ref{tlc}). The chain rule of a convex IVF (Theorem \ref{tcr}) and the $gH$-subgradient of the sum of finite numbers of convex IVFs (Theorem \ref{tss}) have been depicted. Furthermore, the relations between efficient solutions of an IOP with $gH$-subgradient of its objective function have been illustrated (Theorem \ref{tsoc1} and Theorem \ref{tsoc2}).\\

 Although in this article, we have studied various interesting properties of $gH$-subgradients and $gH$-subdifferentials of convex IVFs but could not make any conclusion about nonemptiness of $gH$-subdifferentials. In future, we shall try to make a conclusion about nonemptiness of $gH$-subdifferentials. Also, in connection with the proposed research, future research can evolve in several directions as follows.

\begin{itemize}
    \item The concept of subdifferential of the dual problem of a constrained convex IOP can be illustrated.
    \item A $gH$-subgradient technique to obtain the whole solution set of a nonsmooth convex IOP can be derived.
    \item The derived results can be applied to solve \emph{lasso problem} with interval-valued data.
    \item The notions of quasidifferentiability for IVFs without the help of its parametric representation can be illustrated.
    \item As IVFs are the special case of FVFs and IOPs are the special case of fuzzy optimization problems, similar results can be extended for FVFs and nonsmooth fuzzy optimization problems.

\end{itemize}


\appendix

\section{Proof of norm on $I(\mathbb{R})^n$} \label{apirnnorm}
\begin{proof}
\[
{\lVert \widehat{\textbf{A}} \rVert}_{I(\mathbb{R})^n} = \sqrt{\sum_{i=1}^n {\lVert \textbf{A}_i \rVert}_{I(\mathbb{R})}^2},
\]

\begin{enumerate}[(i)]
\item For any element $\widehat{\textbf{A}}=\left(\textbf{A}_1, \textbf{A}_2, \cdots, \textbf{A}_n\right)\in I(\mathbb{R})^n$, we have
\[
{\lVert \widehat{\textbf{A}} \rVert}_{I(\mathbb{R})^n}= \sqrt{\sum_{i=1}^n {\lVert \textbf{A}_i \rVert}_{I(\mathbb{R})}^2}\geq 0,~\text{since}~{\lVert \textbf{A}_i \rVert}_{I(\mathbb{R})}\geq 0~\text{for all}~i
\]
and
\begingroup\allowdisplaybreaks\begin{align*}
{\lVert \widehat{\textbf{A}} \rVert}_{I(\mathbb{R})^n}=0&\iff {\lVert \textbf{A}_i \rVert}_{I(\mathbb{R})}= 0~\text{for all}~i\\
&\iff \textbf{A}_i = \textbf{0}~\text{for all}~i\\
&\iff \widehat{\textbf{A}}=\widehat{\textbf{0}}=\left(\textbf{0}, \textbf{0}, \cdots, \textbf{0}\right).
\end{align*}\endgroup
\item For any $\lambda\in \mathbb{R}$ and an element $\widehat{\textbf{A}}\in I(\mathbb{R})^n$, we obtain
\begingroup\allowdisplaybreaks\begin{align*}
{\lVert \lambda\odot\widehat{\textbf{A}} \rVert}_{I(\mathbb{R})^n}= \sqrt{\sum_{i=1}^n {\lVert \lambda\odot\textbf{A}_i \rVert}_{I(\mathbb{R})}^2}={\lvert\lambda\rvert}\sqrt{\sum_{i=1}^n {\lVert \textbf{A}_i \rVert}_{I(\mathbb{R})}^2}={\lvert\lambda\rvert}{\lVert\widehat{\textbf{A}} \rVert}_{I(\mathbb{R})^n}.\\
\end{align*}\endgroup
\item For any two elements $\widehat{\textbf{A}}, \widehat{\textbf{B}}\in I(\mathbb{R})^n$, we have
    \[
    {\lVert\widehat{\textbf{A}}\oplus\widehat{\textbf{B}} \rVert}_{I(\mathbb{R})^n}=\sqrt{\sum_{i=1}^n {\lVert \textbf{A}_i \oplus\textbf{B}_i \rVert}_{I(\mathbb{R})}^2}.
    \]
Without loss of generality, due to Definition \ref{irnorm}, let
\[
{\lVert \textbf{A}_i \oplus\textbf{B}_i \rVert}_{I(\mathbb{R})}=\begin{cases}
{\lvert \underline{a}_i +\underline{b}_i \rvert} & \text{for } i=1, 2, \cdots, p(\leq n)\\
{\lvert \overline{a}_i +\overline{b}_i \rvert} & \text{for } i=p+1, p+2, \cdots, n.\\
\end{cases}
\]
Therefore,
\begingroup\allowdisplaybreaks\begin{align*}
&\sqrt{\sum_{i=1}^n {\lVert \textbf{A}_i \oplus\textbf{B}_i \rVert}_{I(\mathbb{R})}^2}\\
=~&\sqrt{\sum_{j=1}^p {\lvert \underline{a}_j +\underline{b}_j \rvert}^2+\sum_{k=p+1}^n {\lvert \overline{a}_k +\overline{b}_k \rvert}^2}\\
\leq~ &\sqrt{\sum_{j=1}^p {\lvert \underline{a}_j \rvert}^2+\sum_{k=p+1}^n {\lvert \overline{a}_k \rvert}^2}+\sqrt{\sum_{j=1}^p {\lvert \underline{b}_j \rvert}^2+\sum_{k=p+1}^n {\lvert \overline{b}_k \rvert}^2}~\text{by Minkowski inequality}\\
\leq~ &\sqrt{\sum_{i=1}^n {\lVert \textbf{A}_i \rVert}^2}+\sqrt{\sum_{i=1}^n {\lVert \textbf{B}_i \rVert}^2}~\text{due to Definition \ref{irnorm}}.
\end{align*}\endgroup
Thus,
\[ {\lVert\widehat{\textbf{A}}\oplus\widehat{\textbf{B}} \rVert}_{I(\mathbb{R})^n}\leq \sqrt{\sum_{i=1}^n {\lVert \textbf{A}_i \rVert}^2}+\sqrt{\sum_{i=1}^n {\lVert \textbf{B}_i \rVert}^2}
\]
\end{enumerate}
Hence, the function ${\lVert \cdot \rVert}_{I(\mathbb{R})^n}$ is a norm on $I(\mathbb{R})^n$.
\end{proof}
\section{Proof of Lemma \ref{lc2}} \label{aplc2}
\begin{proof}
Let $\textbf{F}$ be $gH$-continuous at a point $\bar{x}$ of the set $\mathcal{X}$. Thus, for any $d\in\mathbb{R}^n$ such that $\bar{x}+d\in\mathcal{X}$,
\[
\lim_{\lVert d \rVert\to 0}\left(\textbf{F}(\bar{x}+d)\ominus_{gH}\textbf{F}(\bar{x})\right)=\textbf{0},
\]
which implies
\[
\lim_{\lVert d \rVert\to 0}\left([\underline{f}(\bar{x}+d),\ \overline{f}(\bar{x}+d)]\ominus_{gH}[\underline{f}(\bar{x}),\ \overline{f}(\bar{x})]\right)=[0,\ 0].
\]
Hence, by the definition of $gH$-difference we have
\[
\lim_{\lVert d \rVert\to 0}(\underline{f}(\bar{x}+d)-\underline{f}(\bar{x}))= 0\ \mbox{and}\ \lim_{\lVert d \rVert\to 0}(\overline{f}(\bar{x}+d)-\overline{f}(\bar{x}))= 0,
\]
i.e., $\underline{f}$ and $\overline{f}$ are continuous at $\bar{x}\in\mathcal{X}$.\\

\noindent Conversely, let the functions $\underline{f}$ and $\overline{f}$ be continuous at $\bar{x}\in\mathcal{X}$. If possible, let $\textbf{F}$ be not $gH$-continuous at $\bar{x}$. Then, as $\lVert d \rVert\to 0,\ (\textbf{F}(\bar{x}+d)\ominus_{gH}\textbf{F}(\bar{x}))\not\to\textbf{0}$. Therefore, as $\lVert d \rVert\to 0$ at least one of the functions $\left(\underline{f}(\bar{x}+d)-\underline{f}(\bar{x})\right)$ and $\left(\overline{f}(\bar{x}+d)-\overline{f}(\bar{x})\right)$ does not tend to $0$. Thus, at least one of the functions $\underline{f}$ and $\overline{f}$ is not continuous at $\bar{x}$. This contradicts the assumption that the functions $\underline{f}$ and $\overline{f}$ both are continuous at $\bar{x}$. Hence, $\textbf{F}$ is $gH$-continuous at $\bar{x}$.
\end{proof}
\section{Proof of Theorem \ref{tc}} \label{aptc}
\begin{proof}
Let the IVF $\textbf{F}$ be convex on $\mathcal{X}$. Due to Lemma \ref{lc1}, $\underline{f}$ and $\overline{f}$ are convex on $\mathcal{X}$. Therefore, by the property of real-valued functions, $\underline{f}$ and $\overline{f}$ are continuous on $\mathcal{X}$. Hence, according to Lemma \ref{lc2}, $\textbf{F}$ is $gH$-continuous on $\mathcal{X}$.
\end{proof}
\section{Proof of Theorem \ref{td2}} \label{aptd2}
\begin{proof}
Let the function $\textbf{F}$ be convex on $\mathcal{X}$. Then, for any $x,~y \in \mathcal{X}$ and $\lambda\in (0,1]$, we get
\[
\textbf{F}(x+\lambda (y-x))=\textbf{F}(\lambda y+\lambda'x)~\preceq~ \lambda\odot\textbf{F}(y)\oplus\lambda'\odot\textbf{F}(x),~\text{where}~\lambda'=1-\lambda.
\]
Hence,
\begingroup\allowdisplaybreaks\begin{align*}
\textbf{F}(x+\lambda (y-x))\ominus_{gH}\textbf{F}(x)~\preceq~&(\lambda\odot\textbf{F}(y)\oplus\lambda'\odot\textbf{F}(x))\ominus_{gH}\textbf{F}(x)\\
=~& \big[\lambda \underline{f}(y)+\lambda'\underline{f}(x), \lambda \overline{f}(y)+\lambda'\overline{f}(x)]\ominus_{gH} [\underline{f}(x), \overline{f}(x)\big]\\
	=~& \big[\min \{\lambda \underline{f}(y)+\lambda'\underline{f}(x)-\underline{f}(x), \lambda \overline{f}(y)+\lambda'\overline{f}(x)-\overline{f}(x)  \},\\
	&~~\max \{\lambda \underline{f}(y)+\lambda'\underline{f}(x)-\underline{f}(x), \lambda \overline{f}(y)+\lambda'\overline{f}(x)-\overline{f}(x)  \}\big]\\
	=~&\big[\min\{\lambda\underline{f}(y)-\lambda \underline{f}(x), \lambda\overline{f}(y)-\lambda \overline{f}(x) \}, \\
	&~~\max\{\lambda\underline{f}(y)-\lambda \underline{f}(x), \lambda\overline{f}(y)-\lambda \overline{f}(x) \} \big] \\
	=~&\lambda \odot \big[\min \{\underline{f}(y)-\underline{f}(x), \overline{f}(y)-\overline{f}(x)\},\\
	&~~~~~~~~~\max \{\underline{f}(y)-\underline{f}(x), \overline{f}(y)-\overline{f}(x)\}  \big],~\text{since}~\lambda~>~0\\
	=~&\lambda\odot(\textbf{F}(y)\ominus_{gH}\textbf{F}(x)),
\end{align*}\endgroup
which implies
\[
\frac{1}{\lambda}\odot(\textbf{F}(x+\lambda (y-x))\ominus_{gH}\textbf{F}(x))~\preceq~ \textbf{F}(y)\ominus_{gH}\textbf{F}(x).
\]
Since $\textbf{F}$ is $gH$-differentiable at $x\in\mathcal{X}$, taking $\lambda\to 0+$, by Theorem \ref{td}, we have
\[
(y-x)^T \odot \nabla\textbf{F}(x)~\preceq~ \textbf{F}(y)\ominus_{gH}\textbf{F}(x)~\text{for all}~x,~y\in \mathcal{X}.
\]
\end{proof}
\section{Proof of norm on $\widehat{\textbf{Y}}$} \label{aplfnorm}
\begin{proof}
\begin{enumerate}[(i)]
\item Since $\Vert \textbf{L}(x)\rVert_{I(\mathbb{R})}\geq 0$ and $\lVert x \rVert>0$,

\[
\lVert \textbf{L}\rVert_{\widehat{\mathcal{Y}}} = \sup\limits_{x\neq 0}\frac{\lVert \textbf{L}(x)\rVert_{I(\mathbb{R})}}{\lVert x \rVert}\geq 0~\text{for all}~x\in\mathcal{Y},
\]
and
\begingroup\allowdisplaybreaks\begin{align*}
\lVert \textbf{L}\rVert_{\widehat{\mathcal{Y}}}=0\iff &\sup\limits_{x\neq 0}\frac{\lVert \textbf{L}(x)\rVert_{I(\mathbb{R})}}{\lVert x \rVert}=0\\
\iff & \lVert \textbf{L}(x)\rVert_{I(\mathbb{R})}=0~\text{for all}~x\in\mathcal{Y}\\
\iff & \textbf{L}(x)=\textbf{0}~\text{for all}~x\in\mathcal{Y}\\
\iff & \textbf{L}~\text{is the interval-valued zero mapping;}
\end{align*}
\endgroup
by an interval-valued zero mapping we mean an IVF which maps each element of its domain to $\textbf{0}=[0, 0].$

\item Let $\textbf{L}\in \widehat{\mathcal{Y}}$ and $\lambda \in \mathbb{R}$. Then,
\[
\lVert (\lambda \odot \textbf{L})\rVert_{\widehat{\mathcal{Y}}}=\sup\limits_{x\neq 0}\frac{\lVert (\lambda \textbf{L})(x)\rVert_{I(\mathbb{R})}}{\lVert x \rVert}=\lvert \lambda \rvert \lVert \textbf{L}\rVert_{\widehat{\mathcal{Y}}}.
\]
\item Let $\textbf{L}_1, \textbf{L}_2 \in \widehat{\mathcal{Y}}$. Then,
\begingroup\allowdisplaybreaks\begin{align*}
\lVert \textbf{L}_1  \oplus \textbf{L}_2 \rVert_{\widehat{\mathcal{Y}}} ~=~& \sup\limits_{x\neq 0}\frac{\lVert  \textbf{L}_1 (x) \oplus \textbf{L}_2(x)   \rVert_{I(\mathbb{R})}}{\lVert x \rVert}\\
\leq~& \sup\limits_{x\neq 0}\frac{\Vert \textbf{L}_1(x)\rVert_{I(\mathbb{R})}+ \Vert \textbf{L}_2(x)\rVert_{I(\mathbb{R})}}{\lVert x \rVert}\\
=~&\lVert \textbf{L}_1 \rVert_{\widehat{\mathcal{Y}}}+\lVert \textbf{L}_2 \rVert_{\widehat{\mathcal{Y}}}.
\end{align*}\endgroup
\end{enumerate}
\end{proof}
\section{Proof of Lemma \ref{lnf}} \label{aplnf}
\begin{proof}
For all $x\in\mathcal{Y}$, we have
\begin{equation}\label{en1}
\textbf{L}(x)\preceq \textbf{C}\odot \lVert x \rVert,
\end{equation}
i.e.,
\[
\left[\underline{l}(x), \overline{l}(x)\right]\preceq\left[\underline{c}\lVert x \rVert, \overline{c}\lVert x \rVert\right].
\]
Therefore,
\begin{equation}\label{en2}
\underline{l}(x)\leq\underline{c}\lVert x \rVert~\text{and}~\overline{l}(x)\leq \overline{c}\lVert x \rVert.\\
\end{equation}
Replacing $x$ by $-x$ in the relation (\ref{en1}), we get
\begin{align*}
&\textbf{L}(-x)\preceq \textbf{C}\odot \lVert x \rVert\\
\Longrightarrow~&(-1)\odot\textbf{L}(x)\preceq \textbf{C}\odot \lVert x \rVert\\
\Longrightarrow~&(-1)\odot\textbf{C}\odot \lVert x \rVert\preceq\textbf{L}(x)\\
\Longrightarrow~&\left[-\overline{c}\lVert x \rVert, -\underline{c}\lVert x \rVert\right]\preceq\left[\underline{l}(x), \overline{l}(x)\right],
\end{align*}
which implies
\begin{equation}\label{en3}
\underline{l}(x)\geq-\overline{c}\lVert x \rVert~\text{and}~\overline{l}(x)\geq -\underline{c}\lVert x \rVert.\\
\end{equation}
By the inequalities (\ref{en2}) and (\ref{en3}) we obtain
\begin{align*}
&-\overline{c}\lVert x \rVert\leq\underline{l}(x)\leq\underline{c}\lVert x \rVert~\text{and}~-\underline{c}\lVert x \rVert\leq\overline{l}(x)\leq \overline{c}\lVert x \rVert\\
\Longrightarrow~&\lvert \underline{l}(x) \rvert\leq\max\left\{\lvert \underline{c} \rvert\lVert x \rVert, \lvert \overline{c} \rvert \lVert x \rVert\right\}~\text{and}~\lvert \overline{l}(x) \rvert\leq\max\left\{\lvert \underline{c} \rvert\lVert x \rVert, \lvert \overline{c} \rvert \lVert x \rVert\right\}\\
\Longrightarrow~&\max\left\{\lvert \underline{l}(x) \rvert, \lvert \overline{l}(x) \rvert\right\}\leq\max\left\{\lvert \underline{c} \rvert\lVert x \rVert, \lvert \overline{c} \rvert \lVert x \rVert\right\}\\
\Longrightarrow~&\lVert\textbf{L}(x)\rVert_{I(\mathbb{R})}\leq \lVert\textbf{C}\rVert_{I(\mathbb{R})} \lVert x \rVert
\end{align*}
for all $x\in\mathcal{Y}$.
\end{proof}
\section{Proof of Lemma \ref{lsgbd1}} \label{aplsgbd1}
\begin{proof}
According to Remark \ref{ria1}, without loss of generality, let the first $p$ components of $d$ be nonnegative and the rest $n-p$ components be negative. Therefore, $d^T\odot\widehat{\textbf{A}}$ can be written as
\begingroup\allowdisplaybreaks\begin{align*}
d^T\odot\widehat{\textbf{A}}~=~&\bigoplus_{i=1}^n d_i\odot \textbf{A}_i\\
~=~&\bigoplus_{i=1}^p d_i\odot \textbf{A}_i\oplus\bigoplus_{j=p+1}^{n-p} d_j\odot \textbf{A}_j\\
~=~ & \bigoplus_{i=1}^p \left[\underline{a_{i}}d_i, \overline{a_{i}}d_i\right]\oplus\bigoplus_{j=p+1}^n \left[\overline{a_{j}}d_j, \underline{a_{j}}d_j\right]\\
~=~ & \left[\sum_{i=1}^p\underline{a_{i}}d_i+\sum_{j=p+1}^n\overline{a_{j}}d_j,~ \sum_{i=1}^p\overline{a_{i}}d_i+\sum_{j=p+1}^n\underline{a_{j}}d_j\right].
\end{align*}\endgroup
Therefore,
\begingroup\allowdisplaybreaks\begin{align*}
d^T\odot\widehat{\textbf{A}}\preceq [c, c]~\Longrightarrow~ & \sum_{i=1}^p\underline{a_{i}}d_i+\sum_{j=p+1}^n\overline{a_{j}}d_j \leq \sum_{i=1}^p\overline{a_{i}}d_i+\sum_{j=p+1}^n\underline{a_{j}}d_j\leq c\\
\Longrightarrow~ & w\sum_{i=1}^n\underline{a_{i}}d_i+w'\sum_{i=1}^n\overline{a_{i}}d_i\leq 2c\\
\Longrightarrow~ & \sum_{i=1}^n(w\underline{a_{i}}+w'\overline{a_{i}})d_i\leq 2c\\
\Longrightarrow~ & d^T\mathcal{W}\left(\widehat{\textbf{A}}\right)\leq 2c.
\end{align*}\endgroup
\end{proof}
\section{Proof of Lemma \ref{lsgbd2}} \label{aplsgbd2}
\begin{proof} Let
\[
\lVert\mathcal{W}\left(\widehat{\textbf{A}}\right)\rVert = \sqrt{\left(w\underline{a}_1+w'\overline{a}_1\right)^2+ \left(w\underline{a}_2+w'\overline{a}_2\right)^2+ \cdots + \left(w\underline{a}_n+w'\overline{a}_n\right)^2}.
\]
be finite. Therefore, all $\underline{a}_i$'s and $\overline{a}_i$'s are finite. Hence,
\[
{\lVert \widehat{\textbf{A}} \rVert}_{I(\mathbb{R})^n} = \sqrt{\sum_{i=1}^n {\lVert \textbf{A}_i \rVert}_{I(\mathbb{R})}^2}
= \sqrt{\sum_{i=1}^n \max \left\{\lvert\underline{a_i}\rvert, \lvert\overline{a_i}\rvert\right\}^2}
\]
is finite.
\end{proof}
\section{Proof of Lemma \ref{llc1}} \label{apllc1}
\begin{proof}
Let $x^T\odot\widehat{\textbf{A}}=\textbf{B}$. According to Definition \ref{irnorm}, we have
\[
x^T\odot\widehat{\textbf{A}}=\textbf{B}\preceq \left[\lVert \textbf{B} \rVert_{I(\mathbb{R})}, \lVert \textbf{B} \rVert_{I(\mathbb{R})}\right],
\]
which implies
\[
x^T\odot\widehat{\textbf{A}}\preceq \lVert x \rVert\odot\left[\lVert \widehat{\textbf{A}} \rVert _{I(\mathbb{R})^n}, \lVert \widehat{\textbf{A}} \rVert _{I(\mathbb{R})^n}\right]
\]
because
\begin{align*}
\lVert \textbf{B} \rVert_{I(\mathbb{R})}=~& \lVert x_1\odot \textbf{A}_1\oplus x_2\odot \textbf{A}_2\oplus\cdots\oplus x_n\odot \textbf{A}_n  \rVert_{I(\mathbb{R})} \\
\le~&\lVert x_1\odot \textbf{A}_1 \rVert _{I(\mathbb{R})}+\lVert x_2\odot \textbf{A}_2 \rVert _{I(\mathbb{R})}+\cdots+\lVert x_n\odot \textbf{A}_n \rVert_{I(\mathbb{R})} \\
=~&\lvert x_1 \rvert \lVert \textbf{A}_1\rVert_{I(\mathbb{R})}+\lvert x_2 \rvert \lVert \textbf{A}_2\rVert_{I(\mathbb{R})}+\cdots+\lvert x_n \rvert \lVert \textbf{A}_n\rVert_{I(\mathbb{R})}\\
\leq~& \lVert x \rVert\left(\lVert \textbf{A}_1\rVert_{I(\mathbb{R})}+\lVert \textbf{A}_2\rVert_{I(\mathbb{R})}+\cdots+\lVert \textbf{A}_n\rVert_{I(\mathbb{R})}\right)\\
=~&\lVert x \rVert\lVert \widehat{\textbf{A}} \rVert _{I(\mathbb{R})^n}.
\end{align*}
\end{proof}
\section{Proof of Lemma \ref{llc2}} \label{apllc2}
\begin{proof}
Since $\textbf{F}(x)\ominus_{gH}\textbf{F}(y)\preceq \textbf{C}\odot\lVert x-y \rVert$, for all $x,~y\in \mathcal{X}$, we have
\begin{equation}\label{eggg1}
\underline{f}(x)-\underline{f}(y)\leq c\lVert x-y \rVert~  \text{and}~	 \overline{f}(x)-\overline{f}(y)\leq c\lVert x-y \rVert.
\end{equation}
Interchanging $x$ and $y$ in the inequalities (\ref{eggg1}), we obtain
\begin{equation}\label{eggg2}
\underline{f}(x)-\underline{f}(y)\leq c\lVert x-y \rVert~\text{and}~\overline{f}(x)-\overline{f}(y)\leq c\lVert x-y \rVert~\text{for all}~ x,~y\in \mathcal{X}.
\end{equation}
With the help of the inequalities (\ref{eggg1}) and (\ref{eggg2}), we get
\[
\lvert \underline{f}(x)-\underline{f}(y)\rvert \leq c\lVert x-y \rVert~ \text{and}~\lvert \overline{f}(x)-\overline{f}(y)\rvert \leq c\lVert x-y \rVert~\text{for all}~ x,~y\in \mathcal{X},
\]
which implies
\[
\lVert \textbf{F}(x)\ominus_{gH}\textbf{F}(y)\rVert_{I(\mathbb{R})}\leq c\lVert x-y \rVert~\text{for all}~ x,~y\in \mathcal{X}.
\]
\end{proof}


\end{document}